\documentclass[a4paper,10pt]{article}
\usepackage{mathtext}
\usepackage[T1,T2A]{fontenc}
\usepackage[cp1251]{inputenc}
\usepackage[english]{babel}
\usepackage{amsmath}
\usepackage{amsfonts}
\usepackage{amssymb}
\usepackage{mathrsfs}
\usepackage{amsthm}
\usepackage{titling}

\usepackage{url}
\usepackage{graphicx}
\usepackage{wrapfig}
\usepackage{enumerate}

\usepackage{color}
\usepackage{euscript}

\DeclareMathOperator{\J}{J}
\DeclareMathOperator{\I}{I}
\newcommand{\T}{\mathbb{T}}

\DeclareMathOperator{\interior}{int}

\DeclareMathOperator{\dist}{dist}

\newcommand{\Bell}{\boldsymbol{B}}
\newcommand{\BellC}{\boldsymbol{B}^{\circ}}
\newcommand{\BellCb}{\boldsymbol{B}^{\circ,\mathrm{b}}}

\newcommand{\Bellb}{\boldsymbol{B}^{\mathrm{b}}}
\newcommand{\Class}{\boldsymbol{A}}

\newcommand{\ClassC}{\boldsymbol{A}^{\circ}}

\newcommand{\BMO}{\mathrm{BMO}}

\newcommand{\OMM}{\mathfrak{W}}

\newcommand{\DD}{\mathfrak{D}}

\newcommand{\eps}{\varepsilon}

\DeclareMathOperator{\cl}{cl}

\DeclareMathOperator{\conv}{conv}
\DeclareMathOperator{\E}{\mathbb{E}}
\newcommand{\av}[2]{\langle {#1}\rangle_{{}_{#2}}}

\renewcommand{\leq}{\leqslant}
\renewcommand{\geq}{\geqslant}
\renewcommand{\emptyset}{\varnothing}
\newcommand{\FixedBoundary}{\partial_{\mathrm{fixed}}}
\newcommand{\FreeBoundary}{\partial_{\mathrm{free}}}

\newcommand{\LC}[2]{\Lambda_{{#1},{#2}}}
\newcommand{\scalprod}[2]{\langle{#1},{#2}\rangle}
\newcommand{\per}{\hbox{\tiny \textup{per}}}

\newcommand{\BG}{\mathfrak{B}}
\newcommand{\BM}{\EuScript{B}}
\newcommand{\BMb}{\EuScript{B}^{\mathrm{b}}}
\newcommand{\BMs}{\EuScript{B}^{\mathrm{s}}}

\newcommand{\M}{\EuScript{M}}

\newcommand{\OmNull}{\Omega_0}
\newcommand{\OmOne}{\Omega_1}
\newcommand{\Om}{\Omega}
\newcommand{\tOm}{\tilde{\Om}}

\newcommand{\tOmOne}{\tilde{\Omega}_1}
\newcommand{\OmStar}{\Omega^*}
\newcommand{\hOmOne}{\hat{\Omega}_1}
\newcommand{\hOm}{\hat{\Om}}

\newcommand{\om}{\omega}
\newcommand{\tom}{\tilde{\om}}

\newcommand{\tG}{\tilde{G}}

\newcommand{\R}{\mathbb{R}}

\newcommand{\SSet}[2]{\Bigg\{{#1}\;\Bigg|\,{#2}\Bigg\}}
\newcommand{\Set}[2]{\Big\{{#1}\;\Big|\,{#2}\Big\}}
\newcommand{\set}[2]{\{{#1}\;|\,{#2}\}}

\newcommand{\eq}[1]{\begin{equation}{#1}\end{equation}}
\newcommand{\mlt}[1]{\begin{multline}{#1}\end{multline}}
\newcommand{\alg}[1]{\begin{align}{#1}\end{align}}

\newcommand{\Leqref}[1]{\stackrel{\scriptscriptstyle{\eqref{#1}}}{\leq}}
\newcommand{\Lseqref}[1]{\stackrel{\scriptscriptstyle{\eqref{#1}}}{\lesssim}}

\newcommand{\Lref}[1]{\stackrel{#1}{\leq}}

\newcommand{\F}{\mathcal{F}}

\DeclareMathOperator{\Vis}{Vis}

\newtheorem{Le}{Lemma}[section]
\newtheorem{Def}[Le]{Definition}
\newtheorem{St}[Le]{Proposition}
\newtheorem{Th}[Le]{Theorem}
\newtheorem{Cor}[Le]{Corollary}
\newtheorem{Rem}[Le]{Remark}
\newtheorem{Fact}[Le]{Fact}

\numberwithin{equation}{section}

\begin{document}
\author{Dmitriy~Stolyarov 
\and Pavel~Zatitskiy}
\title{On locally concave functions on simplest non-convex domains
\thanks{Supported by the Russian Science Foundation grant 19-71-10023.}}
\maketitle
\begin{abstract}
We prove that certain Bellman functions of several variables are the minimal locally concave functions. This generalizes earlier results about Bellman functions of two variables. 
\end{abstract}

\section{Introduction}
The aim of the present paper is to extend the theory of~\cite{StolyarovZatitskiy2016}. The main result of that article said that certain Bellman functions of two variables coincide with minimal locally concave functions in the case when their domain is a set theoretic difference of two unbounded convex sets, the smaller lying strictly inside the larger one. The proof relied upon a special class of~$\R^2$-valued martingales and the notion of monotonic rearrangement. We improve these results in two directions: we allow our Bellman functions to depend on more than two variables and also work with the case when the domain is bounded (and therefore, not simply connected in dimension 2). While the special martingales work perfectly in this setting, the notion of monotonic rearrangement is, seemingly, not applicable. We substitute it with the notion of homogenization of a function from~\cite{StolyarovZatitskiy2021}. 

The work is technical: we mostly combine the ideas and methods of two cited papers. Our main results are Theorems~\ref{IntervalTheorem} and~\ref{CircleTheorem}.  Sections~\ref{S2}, \ref{S3}, \ref{S4}, and~\ref{S5} contain definitions, examples, descriptions of previous development of the theory, and statements of the results. Sections~\ref{S6}, \ref{S7}, and~\ref{S8} contain the proofs. We also place two auxiliary results in Section~\ref{S9}.

We wish to thank Vasily Vasyunin for his attention to our work.

\section{Classes of functions}\label{S2}
Let~$\OmNull$ be a non-empty proper open convex  subset of~$\R^d$, here~$d$ is a natural number. Usually~$d \geq 2$. 
Let~$\OmOne$ be another open convex set such that~$\cl\OmOne \subset \OmNull$. It will be convenient to use the notation
\eq{\label{eqdefOmega}
\Om = \cl\OmNull \setminus \OmOne.
}
 We assume~$\Omega_1$ is non-empty for convenience (the case of empty~$\Omega_1$ may be considered by means of classical convex geometry). Let~$\I\subset \R$ be an interval. Consider the class of~$\R^d$ valued summable functions~$\varphi$ defined by the domains~$\OmNull$ and~$\OmOne$:
\eq{\label{Class}
\Class = \Set{\varphi \colon \I\to \partial \OmNull}{\forall\ \J \ \text{subinterval of}\ \I\qquad \av{\varphi}{\J} \notin \OmOne}.
}
Here and in what follows we use the notation
\eq{
\av{\varphi}{E} = \frac{1}{|E|} \int\limits_{E}\varphi(x)\,dx
}
for the average of a summable function~$\varphi$ over a measurable set~$E$ whose Lebesgue measure satisfies the requirement~$0 < |E| < \infty$. Sometimes we will call the points~$\av{\varphi}{\J}$,~$\J$ being a subinterval of~$\I$, the Bellman points of~$\varphi$. Now we will show how several useful classes of functions may be described using particular choices of~$\OmNull$ and~$\OmOne$.

\paragraph{Muckenhoupt classes.} Let~$d=2$ and~$\delta > 1$. We pick particular domains
\eq{\label{MuckenhouptDomain}
\begin{aligned}
\OmNull = \set{(x,y)\in \R^2}{x \geq 0, y \geq 0,\quad xy > 1};\\
\OmOne = \set{(x,y) \in \R^2}{x \geq 0, y \geq 0,\quad xy > \delta}.
\end{aligned}
}
See Fig.~\ref{BasicDomains} for visualization. Consider the class~$\Class$ generated by these domains and a function~$\varphi\in \Class$. Let~$w$ be the first coordinate of~$\varphi$, i.\,e.,~$\varphi(t) = (w(t),w^{-1}(t))$ and~$w \colon \I \to \R_+$ is a scalar almost everywhere positive function. The condition~$\av{\varphi}{\J} \notin \OmOne$
in the definition~\eqref{Class} is rewritten in terms of~$w$ as
\eq{
\av{w}{\J}\av{w^{-1}}{\J} \leq \delta.
}
By our definition, this condition is fulfilled for any interval~$\J \subset \I$ by the requirement~$\varphi \in \Class$. Therefore,~$[w]_{A_2} \leq \delta$ (see Chapter V in~\cite{Stein1993} for definition and basic properties of the Muckenhoupt classes~$A_p$). More specifically, we have proved a simple lemma.
\begin{Le}
The condition~$[w]_{A_2} \leq \delta$ is equivalent to~$\varphi \in \Class$\textup, where the domains~$\OmNull$ and~$\OmOne$ are defined in~\eqref{MuckenhouptDomain} and~$\varphi = (w,w^{-1})$.
\end{Le}
One may prove a similar statement for~$A_p$ classes when~$1 < p < \infty$ and for~$A_{\infty}$ as well, provided the latter class is equipped with Hruschev's norm. The only difference is that one should replace the expression~$xy$ in~\eqref{MuckenhouptDomain} with~$xy^{p-1}$ (and~$x e^{-y}$ in the case~$p=\infty$, see~\cite{Vasyunin2003} for details). See Section~$2$ in~\cite{IOSVZ2015} or Subsection~$1.3$ in~\cite{StolyarovZatitskiy2016} for more information.

\begin{figure}
    \includegraphics[width=0.4\textwidth]{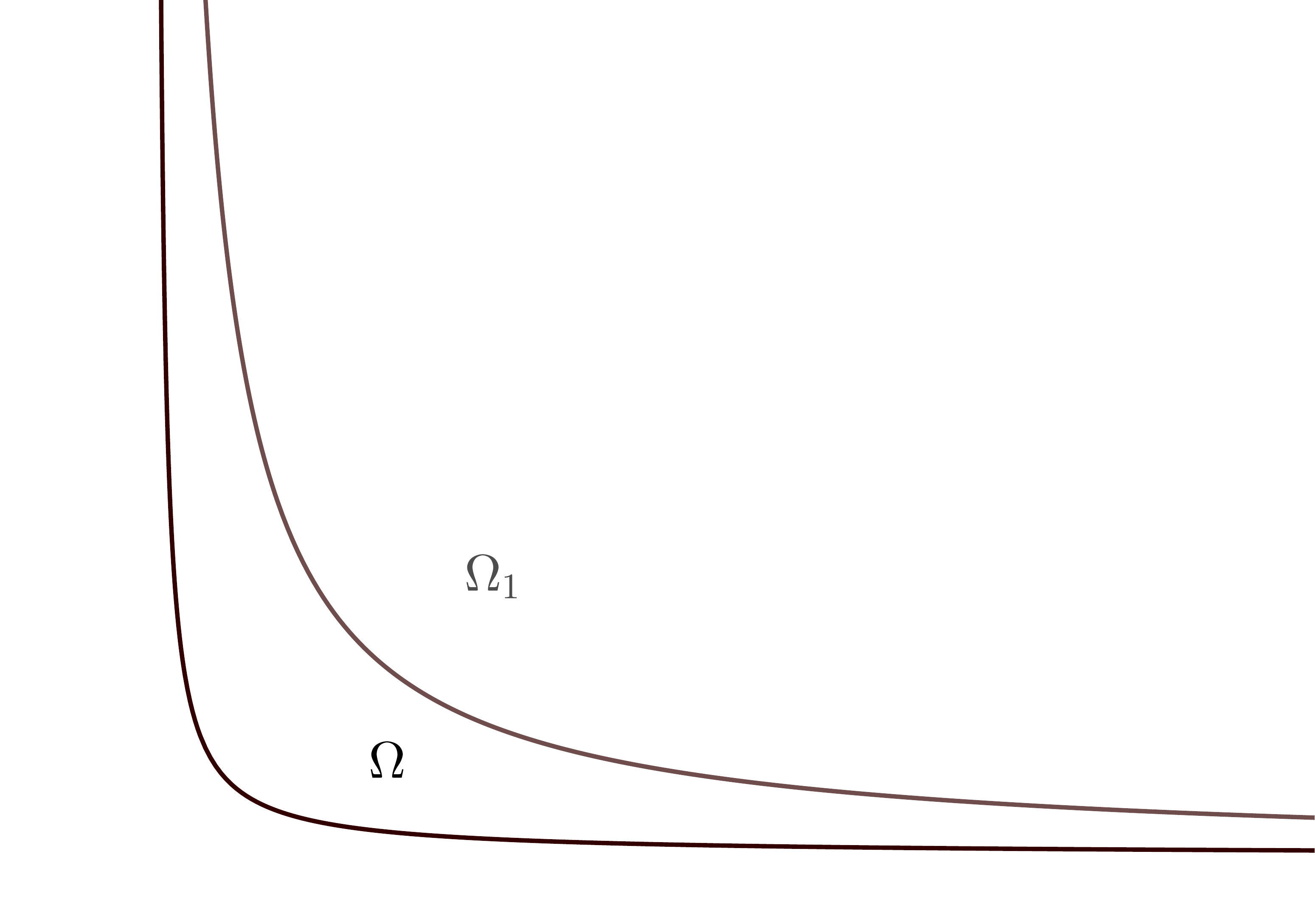}
    \includegraphics[width=0.18\textwidth]{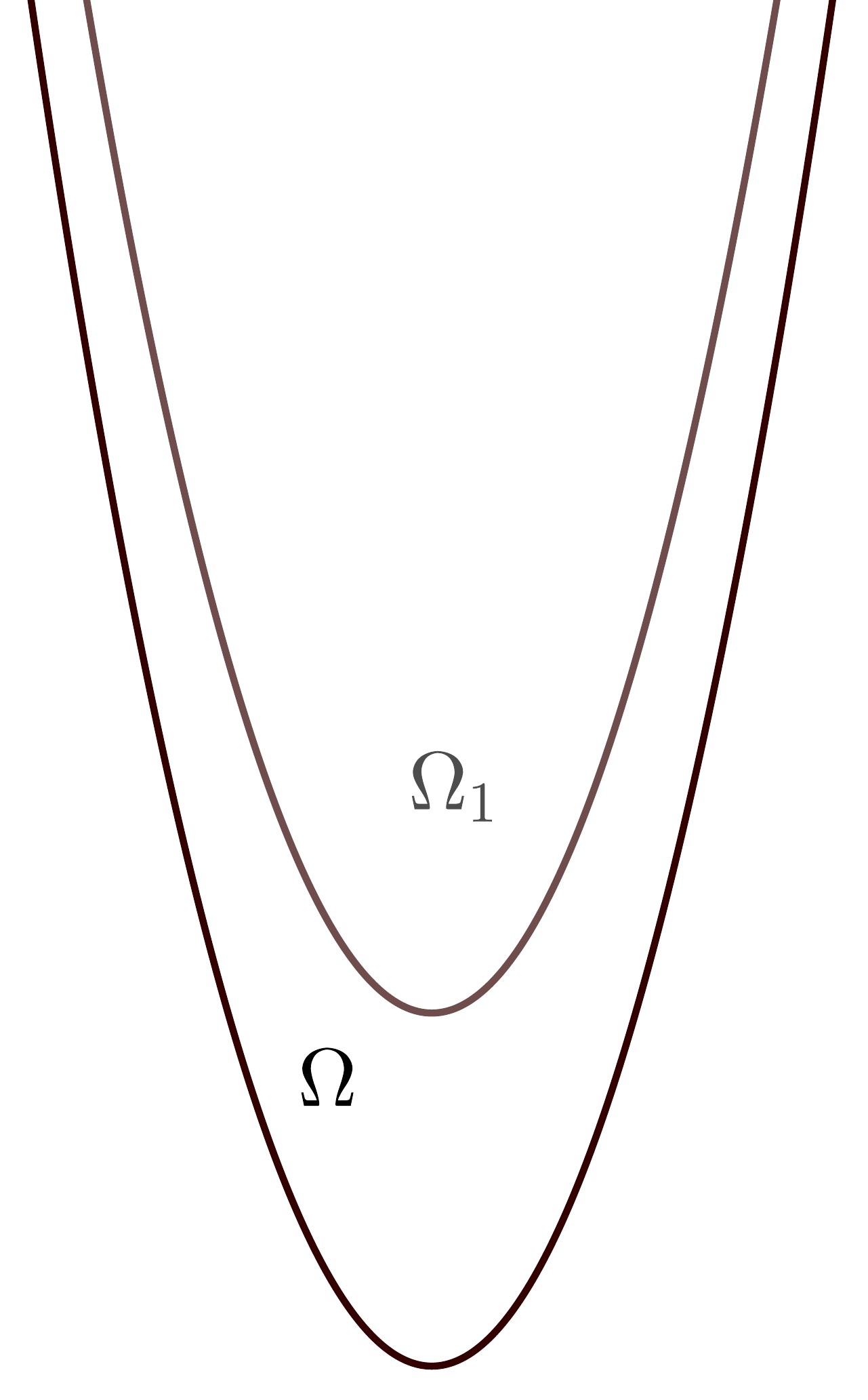}
    \includegraphics[width=0.3\textwidth]{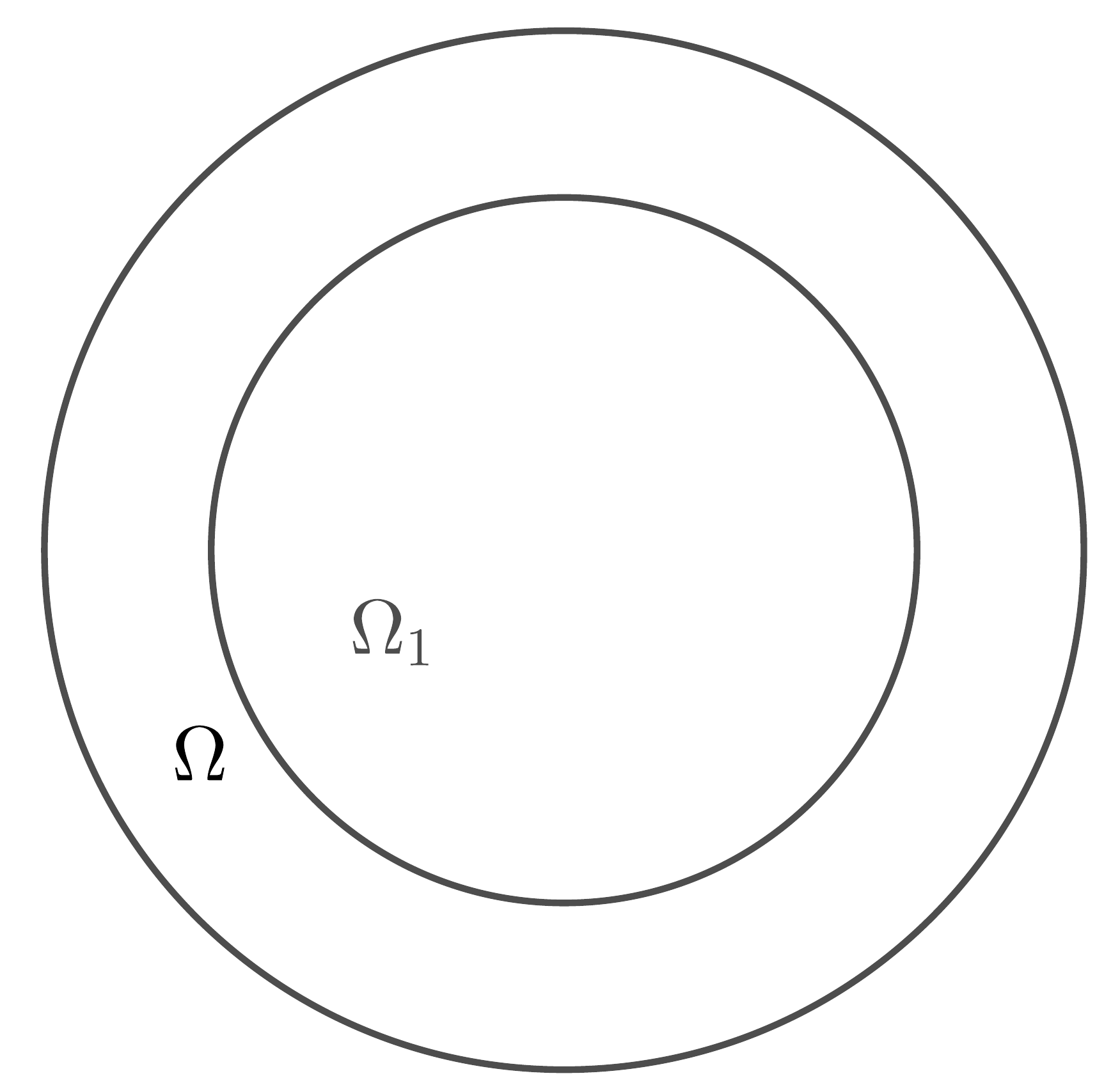}
    \caption{The domains corresponding to~$A_2$, \eqref{MuckenhouptDomain}; scalar-valued~$\BMO_\eps$, \eqref{BMOCase}; and~$\BMO_\eps(\I,S^1)$, \eqref{SphereBMOCase}.}
    \label{BasicDomains}
\end{figure}

\paragraph{Spaces~$\BMO$ of vector-valued functions.} Let~$d$ be an arbitrary natural number larger than one and let~$\eps > 0$. The notation~$|z|$ means the Euclidean norm of~$z \in \R^{d-1}$:
\eq{
|z| = \Big(\sum\limits_{j=1}^{d-1}z_j^2\Big)^\frac12.
} 
Consider the case
\eq{\label{BMOCase}
\begin{aligned}
\OmNull &= \set{(x,y) \in \R^{d-1} \times \R_+}{y > |x|^2};\\
\OmOne &= \set{(x,y) \in \R^{d-1} \times \R_+}{y > |x|^2+ \eps^2}.
\end{aligned}
}
Let~$\psi(t)$ be the vector in~$\R^{d-1}$ formed by the first~$(d-1)$ coordinates of~$\varphi(t)$, where~$\varphi \in \Class$, i.\,e., $\varphi=(\psi,|\psi|^2)$. Then, condition~\eqref{Class} turns into
\eq{
\av{|\psi|^2}{\J} \leq |\av{\psi}{\J}|^2 + \eps^2,
}
which may be rewritten as
\eq{
\av{|\psi - \av{\psi}{\J}|^2}{\J} \leq \eps^2.
}
Since the requirement~$\varphi \in \Class$ means the above inequality holds true for any interval~$\J \subset \I$, we have
\eq{
\|\psi\|_{\BMO(\I)} \leq \eps,
}
provided we define the~$\BMO(\I)$ norm of a vectorial function by the rule
\eq{\label{BMOnorm}
\|\psi\|_{\BMO(\I)} = \Big(\sup\limits_{\J \subset \I} \frac{1}{|\J|}\int\limits_{\J}\Big|\psi(t) - \frac{1}{|\J|}\int\limits_{\J} \psi(s)\,ds\Big|^2\,dt\Big)^{\frac12},
}
where the supremum is taken over all subintervals of~$\I$.
We refer the reader to Chapter IV of~\cite{Stein1993} for the definition and basic properties of the~$\BMO$ space of scalar functions; the quantitative properties of vectorial~$\BMO$ functions are almost the same as that of scalar functions. With this definition at hand, we state yet another simple lemma.
\begin{Le}
The condition~$\|\psi\|_{\BMO(\I)} \leq \eps$ is equivalent to~$\varphi \in \Class$\textup, where the domains are given in~\eqref{BMOCase} and~$\varphi = (\psi,|\psi|^2)$.
\end{Le}
Note that we use the quadratic norm on~$\BMO$ in~\eqref{BMOnorm}. Usually, the definition of~$\BMO$ is given with the more widespread~$L_1$-based norm and after that it is proved via the John--Nirenberg inequality that the two norms are equivalent. Since we will be working with sharp constants, the choice of the particular norm is crucial. 

\paragraph{Functions of bounded mean oscillation with values in the unit sphere.} Let~$d \geq 2$ and let~$\eps \in (0,1)$.
Consider the case
\eq{\label{SphereBMOCase}
\begin{aligned}
\OmNull &= \set{x \in \R^{d}}{|x| < 1};\\
\OmOne &= \set{x \in \R^{d}}{|x|^2 < 1 - \eps^2}.
\end{aligned}
}
Here and in what follows we use the Euclidean norms in~$\R^d$. We see that the functions~$\varphi \in \Class$ attain values in the unit sphere~$S^{d-1}$. Computations similar to those we did in the case of~$\BMO$ functions lead to the following lemma.
\begin{Le}
Let~$\varphi \colon \I\to S^{d-1}$ be a summable function. The condition~$\|\varphi\|_{\BMO} \leq \eps$ is equivalent to~$\varphi \in \Class$\textup, where the domains are given in~\eqref{SphereBMOCase}.
\end{Le}
Following~\cite{BrezisNirenberg1995}, we will call the class of spherically-valued functions whose~$\BMO$ norm does not exceed~$\eps$ the~$\eps$-ball of the space~$\BMO(\I,S^{d-1})$ and denote it by~$\BMO_\eps(\I,S^{d-1})$. Note that~$\BMO(\I,S^{d-1})$ does not have linear structure.

\paragraph{Domain related to multiplicative inequalities.} Here we set~$d=3$. Pick some~$p \in (1,\infty)$ and~$\eps > 0$. Consider the domains
\eq{\label{MultiDomain}
\begin{aligned}
&\OmNull = \interior\conv\set{(t,t^2,|t|^p)}{t\in \R};\\
&\OmOne = \set{(x,y,z) \in \R^3}{y > x^2 + \eps^2, z > 0}.
\end{aligned}
}
The notation~$\conv$ designates the convex hull of a set. Note that these domains do not fulfill the requirement~$\cl\OmOne \subset \OmNull$. They appeared naturally in the study of multiplicative inequalities involving the~$\BMO$ norm in~\cite{SVZ2020} and~\cite{VZZ2021}. In fact, the class~$\Class$ corresponds to the~$\eps$-ball of the~$\BMO$ space. The additional third coordinate allows to keep track of the~$L_p$ norm. This example will be mostly used to show the limitation of our current tools.

\begin{figure}
    \includegraphics[width=0.48\textwidth]{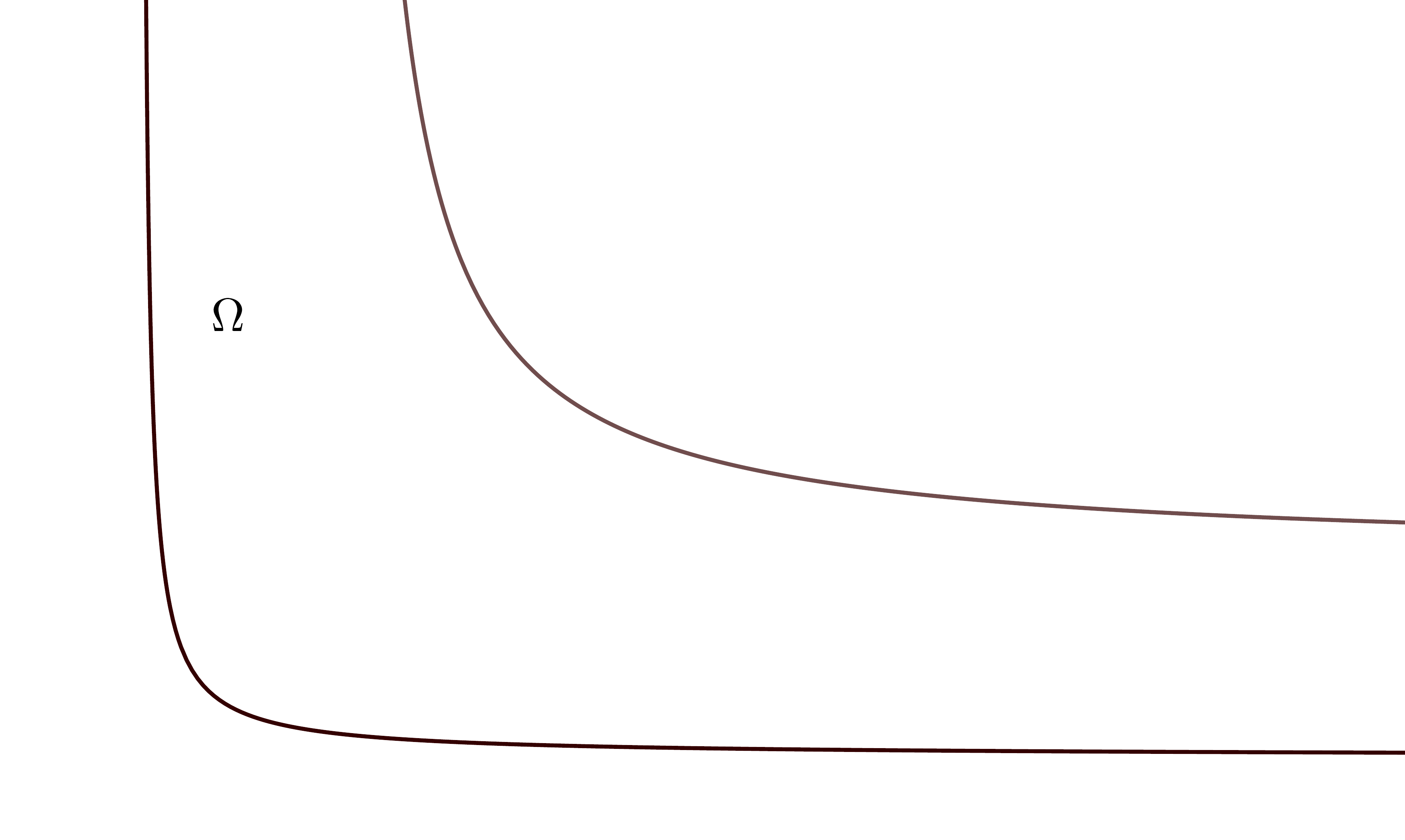}
    \includegraphics[width=0.48\textwidth]{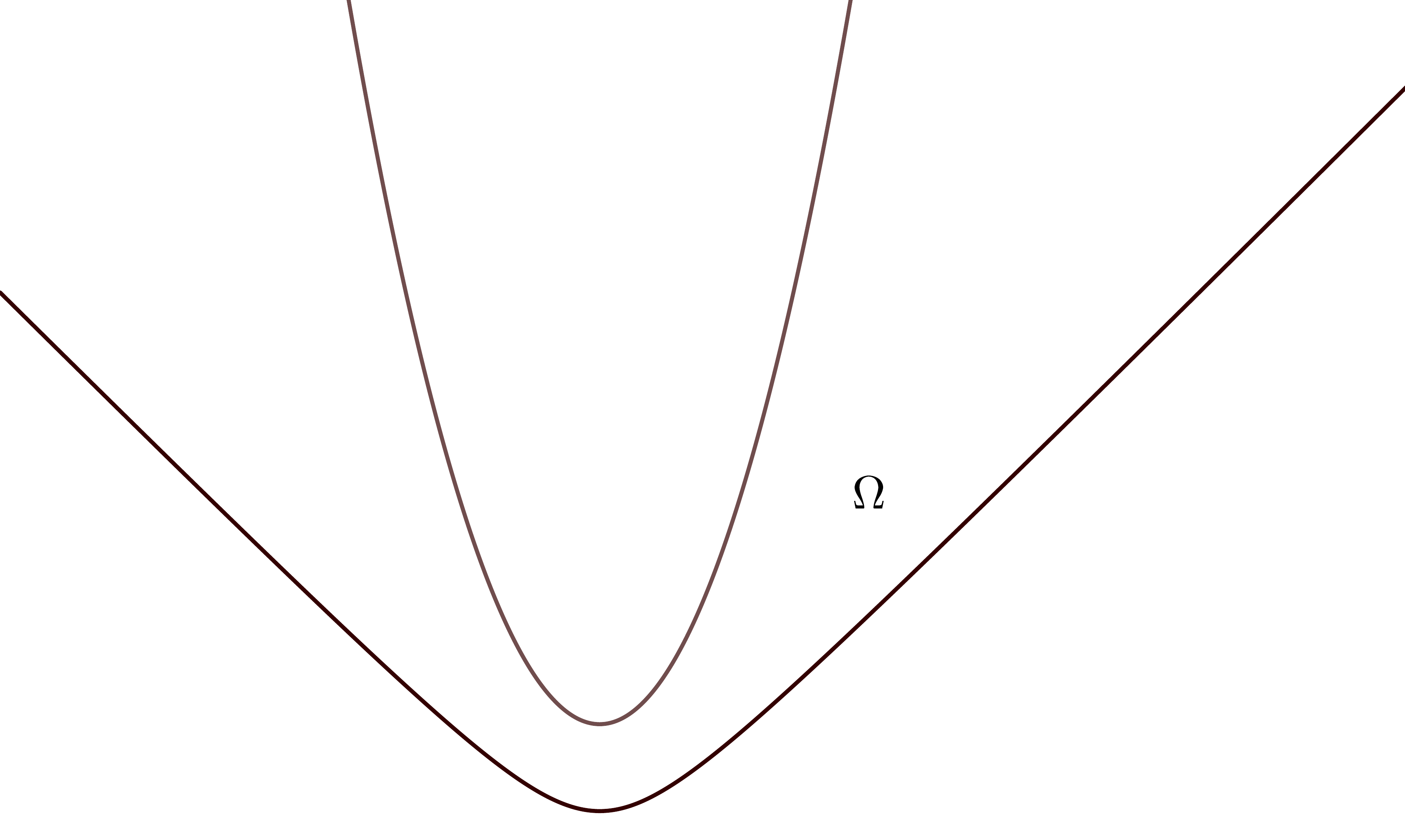}
    \caption{The first domain satisfies the cone condition whereas the second domain does not.}
    \label{RayCondition}
\end{figure}

We need to make further assumptions about the domains~$\OmNull$ and~$\OmOne$. The following two conditions appear naturally in the theory, in particular, the reader may find them in~\cite{StolyarovZatitskiy2016}.

\alg{
\label{StrictConvexity}&\text{{\bf Strict convexity condition:}\ \  the domains~$\OmNull$ and~$\OmOne$ are strictly convex.}\\
\notag\\
\label{ConeCondition} &\text{{\bf Cone condition:}\ \ any ray lying inside~$\OmNull$ has a translate copy that lies inside~$\OmOne$ entirely}.
}

Recall that a convex set is strictly convex provided its boundary does not contain linear segments. Equivalently, a set is strictly convex if and only if any point of its boundary is an exposed point, i.\,e., it is the unique point of intersection of the boundary with a supporting hyperplane. The second assumption somehow says~$\OmNull$ and~$\OmOne$ behave in a similar way at infinity. It may be restated: $\OmNull$ and $\OmOne$ have congruent maximal inscribed cones; this assumption is meaningless if~$\OmNull$ is bounded. The domains on Fig.~\ref{BasicDomains} satisfy conditions~\eqref{StrictConvexity} and~\eqref{ConeCondition} because the corresponding domains $\Omega$ defined in~\eqref{eqdefOmega} do not contain infinite rays. The domain between two shifted hyperbolas (see Fig.~\ref{RayCondition}),
\eq{\label{ExampleWithHyperbolas}
\Om = \Set{(x,y)\in \R^2}{x,y > 0,\quad xy > 1,\ \text{and}\ y<\frac{1}{x-1} + 1\ \text{when}\ x>1}
}
contains infinite rays, e.g., the ones parallel to the coordinate axes. It still satisfies condition~\eqref{ConeCondition}. The second domain on Fig.~\ref{RayCondition},
\eq{
\Om = \Set{(x,y)\in \R^2}{\sqrt{x^2 + 1}\leq y \leq x^2+2},
}
does not satisfy the cone condition. 

The domains of this type (i.\,e., a set theoretic difference of two strictly convex sets, the smaller one lying strictly inside the larger one, and such that~\eqref{ConeCondition} holds true) will be informally called \emph{lenses}. In the proof of the following lemma we will use the notation~$[A,B]$ to denote the straight line segment that connects~$A$ with~$B$.
\begin{Le}\label{DomainLemma}
Assume the domains $\OmNull$ and $\OmOne$ satisfy~\eqref{StrictConvexity}. In the case~$d=2$ we additionally assume~\eqref{ConeCondition}. Then\textup, for any~$x \in \Omega$ there exists a segment~$\ell_x\subset \Omega$ passing through~$x$ and whose endpoints lie on~$\partial \OmNull$.
\end{Le}
\begin{figure}
\begin{center}
\includegraphics[width=0.5\textwidth]{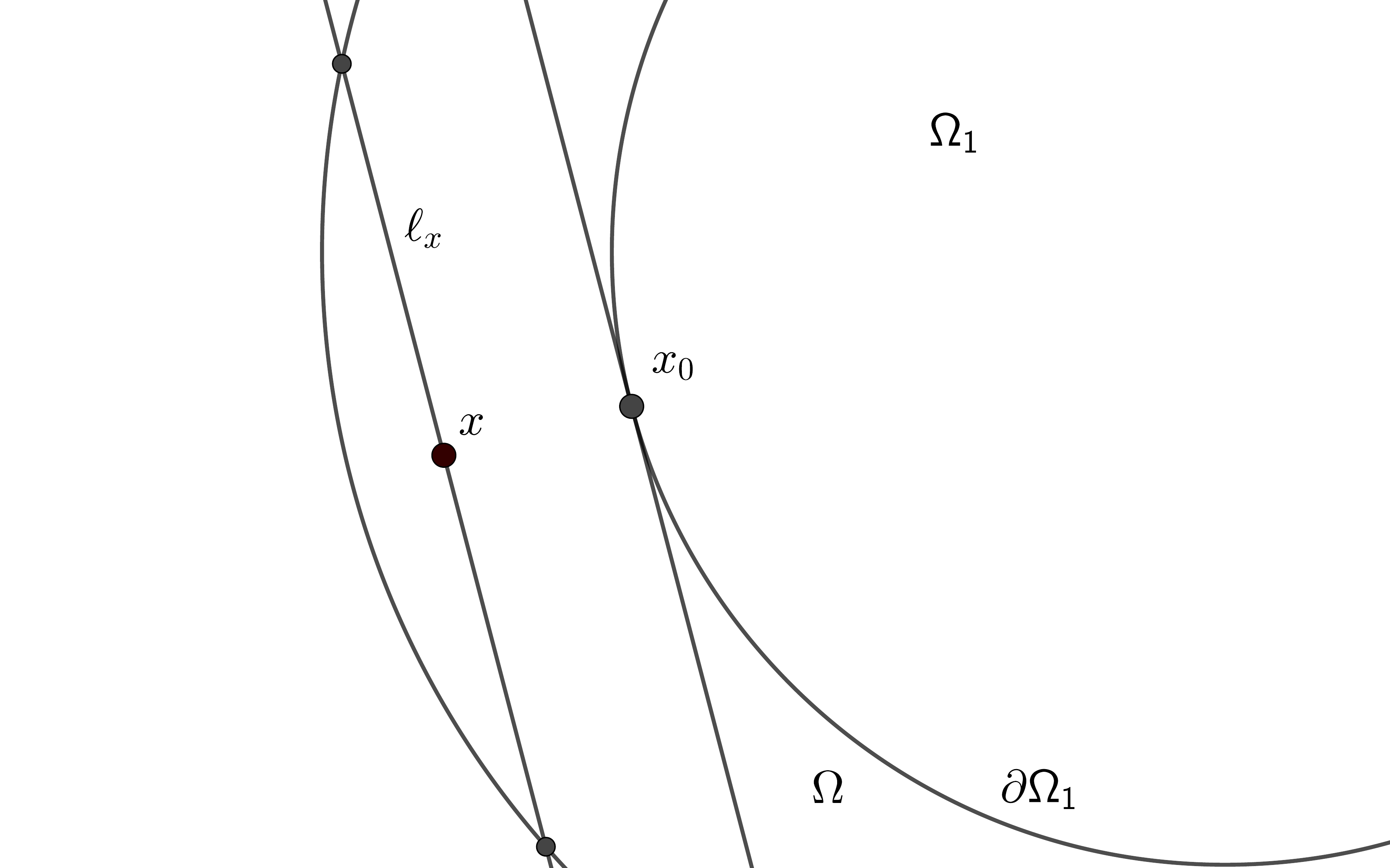}
\end{center}
\caption{Illustration to the proof of Lemma~\ref{DomainLemma}.}
\label{IllLem14}
\end{figure}
\begin{proof}
Consider the case~$d=2$ first. Let~$x_0$ be the closest to~$x$ point in~$\cl\OmOne$, it is clear that~$x_0 \notin \OmOne$. Let~$l$ be a line passing through~$x_0$ that is perpendicular to~$[x,x_0]$ (if $x=x_0 \in \partial\OmOne$, we take $l$ to be a supporting line to~$\cl\OmOne$ at $x_0$). Note that~$l$ does not intersect~$\OmOne$ and separates~$x$ from~$\OmOne$. By~\eqref{StrictConvexity}, the intersection of any translate of~$l$ with~$\OmOne$ is bounded. Thus, by~\eqref{ConeCondition}, the intersection of the translate of~$l$ passing through~$x$ with~$\OmNull$ is a finite segment, let it be~$\ell_x$. It remains to note that~$\ell_x \cap \OmOne = \emptyset$. 

Now let us turn to the case~$d \geq 3$. We will argue by induction over dimension. In fact, it suffices to show that for any~$x\in \Om$ there exists an affine hyperplane~$L$ passing through~$x$ such that~$L\cap \OmNull$ is bounded. If this assertion is proved, we may work inside the~$(d-1)$ dimensional plane~$L$ (condition~\eqref{ConeCondition} holds true for bounded domains). The desired hyperplane~$L$ is also easy to find: pick any supporting hyperplane to~$\OmNull$ and translate it to~$x$ (the intersection of this translated supported hyperplane and~$\OmNull$ is compact by~\eqref{StrictConvexity}).
\end{proof}
\begin{Rem}\label{DomainRemark}
In the assumptions of Lemma~\textup{\ref{DomainLemma},} if~$x\notin \cl \OmOne$\textup, then~$\ell_x$ may be chosen in such a way that~$\ell_x \cap \cl\OmOne = \emptyset$ as well.
\end{Rem}
\begin{Cor}\label{CorollaryOfDomainLemma}
In the assumptions of Lemma~\textup{\ref{DomainLemma},} for any~$x \in \Omega$ there exists~$\varphi \in \Class$ such that~$\av{\varphi}{\I} = x$.
\end{Cor}
\begin{proof}
Construct a segment~$\ell_x$ with the help of Lemma~\ref{DomainLemma}: there exist points~$a,b\in \partial \OmNull$ and non-negative numbers~$\alpha$ and~$\beta$ with sum one such that
\eq{
x = \alpha a + \beta b \quad \text{and} \quad [a,b]\cap \OmOne=\varnothing.
}
Then, the desired function~$\varphi$ may be constructed as
\eq{\label{StepFunction}
\varphi(t) = \begin{cases}
a,\quad t \in [0,\alpha];\\
b,\quad t \in (\alpha,1],
\end{cases}
}
here we set~$\I = [0,1]$ for convenience (all the considerations do not depend on the particular choice of~$\I$). By construction, for any~$\J \subset [0,1]$, the point~$\av{\varphi}{\J}$ lies inside~$\ell_x$. Therefore,~$\varphi \in \Class$.
\end{proof}

\bigskip

\section{Functionals}\label{S3}
Let~$f\colon \partial \Omega_0 \to \mathbb{R}$ be a Borel measurable locally bounded function. We wish to find sharp estimates of the expression~$\av{f(\varphi)}{\I}$ when~$\varphi \in \Class$. Note that a priori it is unclear whether the integral in question exists. 
The function
\eq{\label{Bellman}
\Bell_{\Omega,f}(x) = \sup\Set{\av{f(\varphi)}{\I}}{\varphi \in \Class,\ \av{\varphi}{\I} = x},\quad x\in \Omega,
}
is well defined in the case where~$f$ is bounded from below (though this function may attain the value~$+\infty$). The function
\eq{\label{Bellmanb}
\Bellb_{\Omega,f}(x) = \sup\Set{\av{f(\varphi)}{\I}}{\varphi \in \Class,\ \av{\varphi}{\I} = x, \quad \varphi \in L_{\infty}},\qquad x\in \Omega,
}
is a meaningful object for any~$f$ locally bounded from below. Note that all the Bellman functions in the paper do not depend on the choice of $I$, because the classes of functions~$\Class$ (see~\eqref{Class}) and the Bellman functions themselves are defined in terms of averages. We will often omit the symbols~$\Omega$ and~$f$ in the notation for our Bellman functions and simply write~$\Bell$ and~$\Bellb$. We also use the notation~$\OmStar = \cl\OmNull\setminus \cl\OmOne$. Of course, we expect that the functions~$\Bell$ and~$\Bellb$ coincide in reasonable situations. However, the proof of this assertion might be unexpectedly difficult. 
\begin{Rem}
Assume~\eqref{StrictConvexity} holds true. In the case~$d=2$ we also require~\eqref{ConeCondition}. Corollary~\textup{\ref{CorollaryOfDomainLemma}} says that in this case
\eq{
-\infty < \Bellb(x) \leq \Bell(x),\qquad x\in \Omega.
}
\end{Rem}
Now we pass to the examples and show how the Bellman functions above help to find sharp constants in various inequalities. We do not provide many details for the first two examples since they are discussed in the cited papers. 

\paragraph{Muckenhoupt classes.} Consider the domains~\eqref{MuckenhouptDomain} and the function~$f(x_1,x_1^{-1}) = x_1^p$, where~$p > 1$. The computation of the corresponding Bellman function~$\Bell$ leads to sharp constants in various forms of the Reverse H\"older inequality for Muckenhoupt weights. See~\cite{Vasyunin2004} and~\cite{Vasyunin2009} for details. For weak-type bounds, one makes the choice~$f(x_1,x_1^{-1}) = \chi_{[1,\infty)}(x_1)$, see~\cite{Reznikov2013}.

\paragraph{Scalar~$\BMO$ space.} Consider the domains~\eqref{BMOCase} with~$d=2$. The choice~$f(x_1,x_1^2) = e^{x_1}$ leads to the sharp John--Nirenberg inequality in integral form, see~\cite{SlavinVasyunin2011}. The function~$f(x_1,x_1^2) = |x_1|^p$ was used to obtain sharp results on the constants in the inequalities that express the equivalence of different norms on~$\BMO$, see~\cite{SlavinVasyunin2012}. For weak-type estimates, we may choose the function~$f(x_1,x_1^2) = \chi_{[0,\infty)}(x_1)$, see~\cite{VasyuninVolberg2014}. For quite general boundary values~$f$, the function~\eqref{Bellman} was computed in~\cite{ISVZ2018} (see a simpler version~\cite{IOSVZ2016} and the short report~\cite{IOSVZ2012}).

For larger~$d$ similar Bellman functions will lead to sharp inequalities for vectorial functions. 

\paragraph{Functions of bounded mean oscillation with values in the unit sphere.}
A version of the John--Nirenberg inequality for~$\BMO$ functions between manifolds may be found in Appendix B of~\cite{BrezisNirenberg1995}. In that paper the inequality is stated in its integral form, here we prefer to work with the classical 'tail estimate' form as in~\cite{JohnNirenberg1961}. For that consider the class~$\Class$ generated by the domains~$\OmNull$ and~$\OmOne$ given in~\eqref{SphereBMOCase}. Pick some point~$x_0 \in S^{d-1}$ and some~$\delta \in (0,1)$. Consider the function
\eq{
f(x) = \chi_{[\delta,2]}(|x-x_0|),\quad x\in S^{d-1},
}
and the Bellman function~\eqref{Bellman} generated by this boundary value. The Bellman function delivers sharp estimate of the amount of points~$t\in \I$ such that
\eq{
|\varphi(t) - x_0| \geq \delta,
}
provided~$\varphi$ belongs to the~$\eps$-ball of~$\BMO(\I,S^{d-1})$ and~$\av{\varphi}{\I} = x$. If we choose~$x = x_0 |x|$, we obtain the sharp estimate for the quantity
\eq{
\frac{1}{|\I|}\Big|\Set{t\in \I}{|\varphi(t) - \av{\varphi}{\I}| \geq \tilde{\delta}}\Big|; \qquad \tilde{\delta}^2 = (1-|x|)^2+|x|\delta^2.
}
The John--Nirenberg inequality says this quantity decays exponentially as~$\eps$ decreases down to zero. The Bellman function allows to find the sharp constants in the corresponding inequality, see the forthcoming paper~\cite{Dobronravov2021}.

\paragraph{Multiplicative inequalities.} 
Consider the domains given by~\eqref{MultiDomain} and the corresponding class~$\Class$. Let~$f(t,t^2,|t|^p) = |t|^r$, where~$r\in (p,\infty)$. The corresponding Bellman function delivers sharp upper estimates for the~$L_r$-norm of a function~$\varphi$ provided its average,~$L_p$-norm, and~$\BMO$ norms are fixed. This, in particular leads to computation of sharp constants~$c_{p,r}$ in the multiplicative inequalities
\eq{
\|\varphi\|_{L_r} \leq c_{p,r}\|\varphi\|_{L_p}^{p/r}\|\varphi\|_{\BMO}^{1-p/r},
}
see~\cite{SVZ2020} and~\cite{VZZ2021}.

\section{Locally concave functions and martingales}\label{S4}
\begin{Def}
Let~$\omega \subset \R^d$. We say that a function~$G \colon\omega \to \R \cup \{+\infty\}$ is locally concave provided its restriction~$G|_\ell$ to any segment~$\ell \subset \omega$ is concave.
\end{Def}
Here and in what follows we will be using the convention that concave functions may attain infinite values, see~\cite{Rockafellar1970}. It is important that we do not allow the value~$-\infty$ (see a pathological example at the end of this section). Locally concave functions play important role in the Bellman function theory; see~\cite{Guan1998} for applications to geometry. In the definition below,~$\FixedBoundary\omega$ is the set of all points~$x\in \omega$ such that there does not exist a segment~$\ell\subset \omega$ with~$x$ lying in the interior of~$\ell$; the latter set is called the fixed boundary (because we fix the boundary values of locally concave functions on this set). The remaining part of the boundary is called the free boundary and is denoted by~$\FreeBoundary\omega$. In our usual examples of a lens~$\Om = \cl \OmNull\setminus \OmOne$, we have~$\FixedBoundary\Omega = \partial \OmNull$ and $\FreeBoundary\Omega = \partial \OmOne$. From now on we will be using this notation.

\begin{Def}
Let~$\omega \subset \R^d$\textup, let~$f\colon \FixedBoundary\omega\to \R$ be a function. By~$\LC{\omega}{f}$ we denote the class of all locally concave on~$\omega$ functions~$G$ that satisfy the boundary inequality~$G(x) \geq f(x)$ for any~$x\in \FixedBoundary\omega$.
\end{Def}
\begin{Rem}
By Theorem~$10.1$ in~\textup{\cite{Rockafellar1970},} any function~$G\in \LC{\omega}{f}$ is continuous on the interior of~$\omega$ as a mapping with values in~$\R\cup \{+\infty\}$. 
\end{Rem}
We are ready to define the pointwise minimal locally concave function
\eq{\label{DefinitionOfMinimalLocallyConcave}
\BG_{\omega,f}(x) = \inf\set{G(x)}{G \in \LC{\omega}{f}},\quad x\in \omega,
}
(note that~$\BG_{\omega,f}(x) \in \LC{\omega}{f}$ if this function does not attain the value~$-\infty$) and state our first main theorem. The notation~$\FreeBoundary\Om \in C^2$ means that locally~$\FreeBoundary \Om$ coincides with a graph of a~$C^2$-smooth function. Recall~$\OmStar = \cl\OmNull \setminus \cl\OmOne$.
\begin{Th}\label{IntervalTheorem}
Let the domains~$\OmNull$ and~$\OmOne$ satisfy the usual requirements~$\cl \OmOne \subset \OmNull$\textup,~\eqref{StrictConvexity}\textup,~\eqref{ConeCondition}\textup, and let also~$\FreeBoundary \Om \in C^2$. Let the function~$f$ be lower semi-continuous and bounded from below. Then\textup,
\eq{\label{CoincidenceFormula}
\Bell_{\Omega,f}(x) = \BG_{\Omega, f}(x),\qquad x\in\OmStar.
}
\end{Th}
This theorem generalizes, up to minor modifications, the main result of~\cite{StolyarovZatitskiy2016}; in that paper we had~$d=2$ and unbounded~$\OmOne$ (and, therefore, unbounded~$\OmNull$). The methods of~\cite{StolyarovZatitskiy2016} relied upon the notion of monotonic (non-decreasing) rearrangement, which, seemingly, does not exist in the case where~$\Omega$ is not simply connected. A good replacement for monotonic rearrangements was found in~\cite{StolyarovZatitskiy2021}. The disadvantage of this new method is that it does not allow to work with~$x\in \FreeBoundary\Omega$. The main theorem in~\cite{StolyarovZatitskiy2016} states identity~\eqref{CoincidenceFormula} for all~$x\in \Omega$. Seemingly, in the larger generality Theorem~\ref{IntervalTheorem} is also true for~$x\in \Omega$, however, the methods from~\cite{StolyarovZatitskiy2021} do not allow to prove it (see Section~\ref{S9} below for an explanation). 

We need to survey the notions from~\cite{StolyarovZatitskiy2016} we will be using. We will be working with discrete time martingales adapted to a filtration~$\F = \{\F_n\}_n$. We refer the reader to~\cite{Shiryaev2019} for the general martingale theory and present here the simplified definitions we will use. 

By a filtration we mean a sequence of increasing set algebras (we consider finite algebras only), that is if~$A \in \F_n$, then~$A \in \F_{n+1}$ as well. A sequence~$\{M_n\}_n$ of random variables taking values in some linear space is called a martingale adapted to~$\F$~provided, first, each~$M_n$ is~$\F_n$-measurable, and second, for each~$n$ we have~$M_n = \E(M_{n+1}\mid \F_n)$. Note that since our algebras are simple, we may freely work in infinite dimensional spaces (in the case of general martingales, there are difficulties with the definition of conditional expectation, see Section~$1.3$ in~\cite{HNVW2016}). All martingales we will be working with have the limit value~$M_\infty\in L_1$, a random variable that generates the martingale:
\eq{
M_n = \E(M_\infty\mid \F_n),\qquad n \in \mathbb{N}\cup \{0\}.
}
Here we should take care,~$M_\infty$ should attain values in a finite dimensional space since we wish to omit the theory of integration of functions taking values in infinite-dimensional spaces. 

The main property of~$M_\infty$ may be restated: for any atom~$a\in \F_n$ (by an atom of a set algebra~$\F$ we mean a set~$a\in \F$ of positive measure that is minimal by inclusion), one may restore the value of~$M_n$ on this atom by the formula
\eq{
M_n(a) = \frac{1}{P(a)}\int\limits_{a} M_{\infty}.
}
We cite an important definition from~\cite{StolyarovZatitskiy2016}. See Fig.~\ref{ExampleOmMart} for an example.
\begin{Def}\label{OmegaMart}
Let~$\omega \subset \R^d$. An~$\R^d$-valued martingale~$M = \{M_n\}_n$ adapted to~$\F = \{\F_n\}$ is called an~$\omega$-martingale\textup, provided
\begin{enumerate}[1\textup)]
\item the algebra~$\F_0$ is trivial\textup, i.\,e.\textup, consists of the whole probability space and the empty set\textup;
\item there exists a random variable~$M_\infty$ with the values in $\FixedBoundary \omega$ such that~$M_n\to M_{\infty}$ in~$L_1$ and almost everywhere \textup(in particular\textup,~$M_{\infty}$ is summable itself\textup)\textup;
\item for any atom~$a \in \F_n$ there exists a convex set~$C_a\subset \omega$ such that~$M_{n+1}|_a$ lies in~$C_a$ almost surely.
\end{enumerate}
\end{Def}
Sometimes, we will need to use a slightly modified notion of an~$(\omega, \mathfrak{D})$-martingale introduced in~\cite{StolyarovZatitskiy2021}.
\begin{Def}\label{OmegaDeltaMart}
Let~$\omega\subset \R^d$ and let~$\mathfrak{D} \subset \FixedBoundary\omega$. An~$\omega$-martingale~$M$ is called an~$(\omega,\mathfrak{D})$-martingale\textup, provided~$M_\infty \in \mathfrak{D}$ almost surely.
\end{Def}
\begin{Rem}\label{InfDimRem}
We may also consider~$\omega$ or~$(\omega, \mathfrak{D})$-martingales for infinite-dimensional domains~$\omega$. However\textup, in this case we require that~$M_\infty$ attains values in the intersection of~$\FixedBoundary \omega$ with a finite dimensional space. For example\textup, this happens if~$M_\infty$ attains finitely many values.
\end{Rem} 
\begin{figure}\label{ExampleOmMart}
\begin{center}
   \includegraphics[width = 0.35 \textwidth]{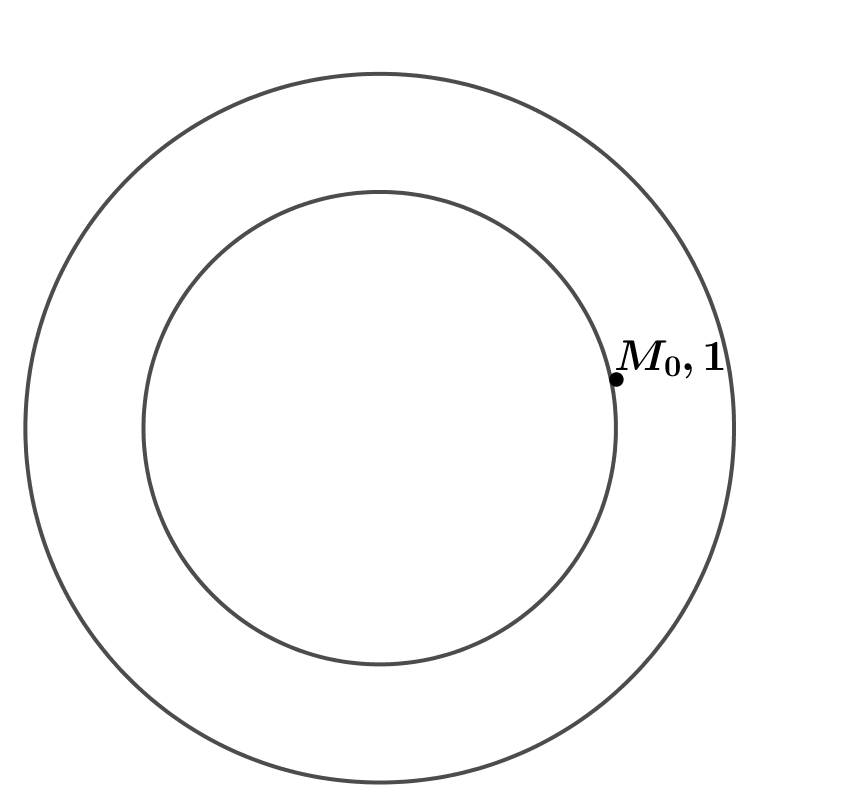}\hspace{-20pt}
   \includegraphics[width = 0.35 \textwidth]{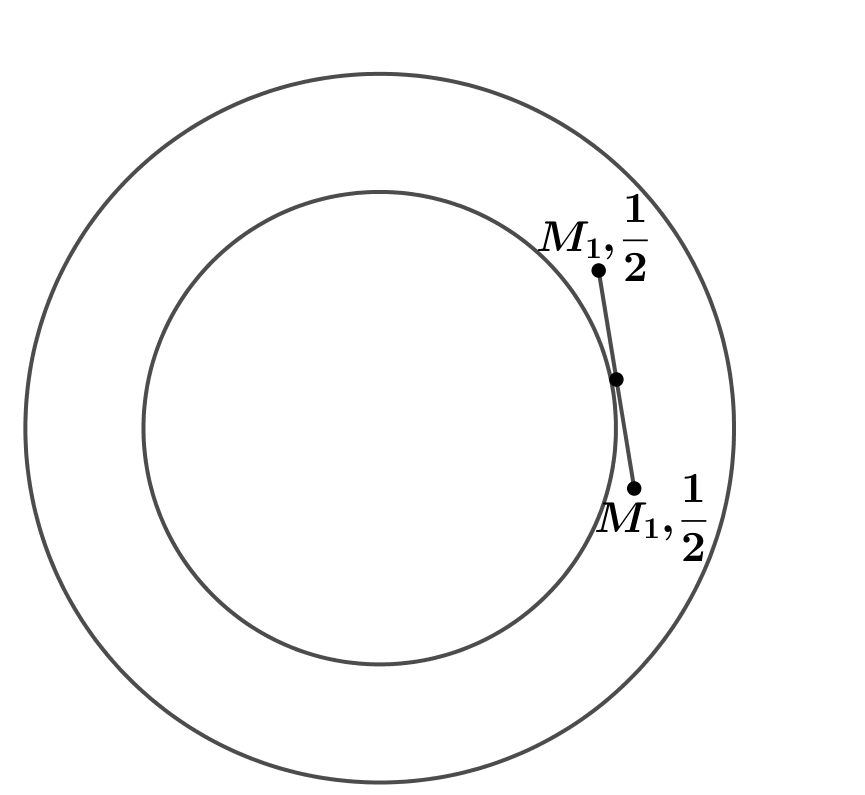}\hspace{-20pt}
   \includegraphics[width = 0.35 \textwidth]{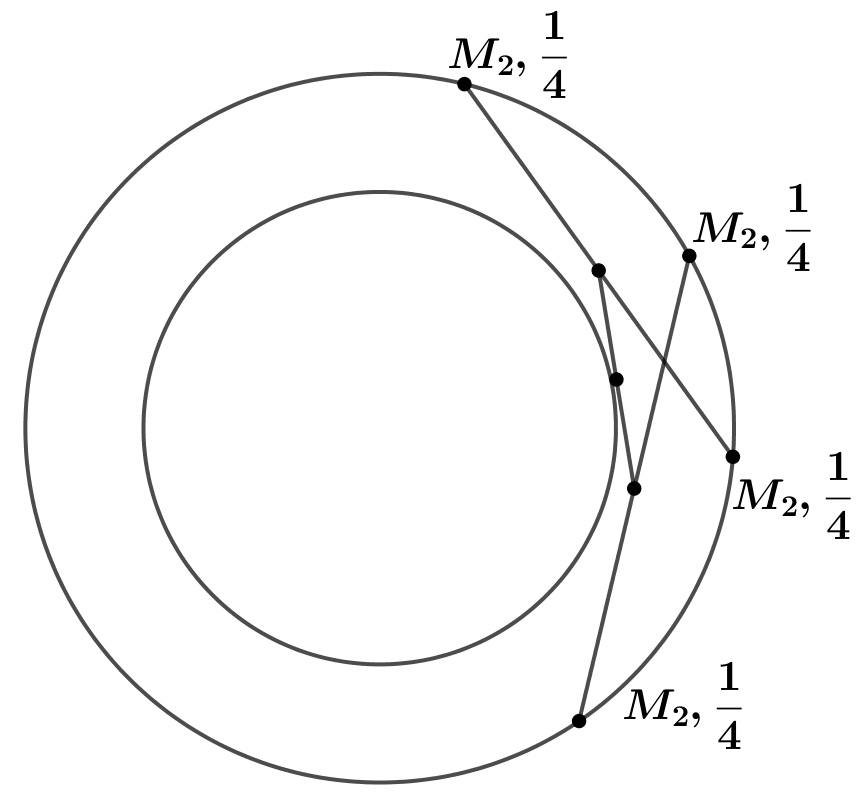}
\end{center}
    \caption{An example of an~$\omega$-martingale where~$\omega$ is given in~\eqref{SphereBMOCase}; the numbers assigned to points denote the probability of the specific value.}
\end{figure}

Consider two other Bellman functions on~$\omega$: the first one is defined for~$f\colon \FixedBoundary\omega \to \R$ measurable and bounded from below,
\eq{
\BM(x) = \sup\Set{\E f(M_{\infty})}{M_0 = x,\ M\ \text{is an~$\omega$-martingale}},\quad x\in \omega,
}
and the second one is defined for arbitrary locally bounded from below measurable~$f$:
\eq{
\BMb(x) = \sup\Set{\E f(M_{\infty})}{M_0 = x,\ M\ \text{is an~$\omega$-martingale},\ M_{\infty}\in L_{\infty}},\quad x\in \omega.
} 
Recall that a martingale~$M$ is called simple if there exists a natural number~$N$ such that~$M_k = M_{N}$ for all~$k \geq N$. 
We will use Theorem~$2.21$ from~\cite{StolyarovZatitskiy2016}. It uses the notion of a \emph{strongly martingale connected domain}. A set~$\omega \subset \R^d$ is called strongly martingale connected if for every~$x\in \omega$ there exists a simple~$\omega$-martingale~$M$ that starts at~$x$, that is~$M_0 = x$. 
\begin{Le}\label{Le52}
Assume~\eqref{StrictConvexity}. The domains~$\Omega$ and~$\Omega^*$ are strongly martingale connected if either~$d \geq 3$ or~\eqref{ConeCondition} holds true.
\end{Le}
\begin{proof}
Consider the case of the domain~$\Omega$ first. Let~$x\in \Omega$, we wish to construct a simple~$\Omega$-martingale starting at~$x$. If~$x\in \FixedBoundary \Omega$, then the desired martingale is constantly equal~$x$. So, we may assume~$x\notin \FixedBoundary\Omega$. By Lemma~\ref{DomainLemma}, there exists a segment~$\ell_x\subset \Om$ with the endpoints~$A$ and~$B$ lying on~$\FixedBoundary\Omega$, passing through~$x$. In other words,
\eq{
x = \alpha A + \beta B,\qquad \alpha + \beta = 1,\ \text{and}\ \alpha, \beta \geq 0.
}
Let us construct the martingale~$M$ by the formula,
\eq{
M_0 = x,\quad M_1 = \begin{cases}
A,\quad \text{with probability}\ \alpha;\\
B,\quad \text{with probability}\ \beta,
\end{cases}
M_n = M_1,\quad n \geq 1.
}
In other words,~$M$ is a martingale that starts at~$x$, splits into~$A$ and~$B$, and stops there. Since~$\ell_x \subset \Omega$,~$M$ is the desired simple~$\Omega$-martingale. 

The reasoning for the domain~$\Omega^*$ is completely similar. One relies upon Remark~\ref{DomainRemark} instead of Lemma~\ref{DomainLemma} in this case. 
\end{proof}

\begin{Th}[Theorem~$2.21$ in~\cite{StolyarovZatitskiy2016}]\label{Th53}
Let~$\omega \subset \R^d$ be a strongly martingale connected domain. Let $f\colon \FixedBoundary\omega\to \R$ be a bounded from below function such that~$\BG_{w,f}$ is continuous at every point of the fixed boundary. Then\textup,~$\BG(x) = \BM(x)$ for all~$x\in \omega$. 
\end{Th}
Remark~$2.22$ of the same paper says that if~$f$ is only locally bounded from below,~$\BG_{\omega,f}$ is continuous at the points of the fixed boundary, then~$\BG=\BMb$. We will need a slightly stronger statement. Let us introduce yet another Bellman function
\eq{
\BMs(x) = \sup\Set{\E f(M_{\infty})}{M_0 = x,\ M\ \text{is a simple~$\Omega$-martingale}},\quad x\in \omega.
} 
Clearly,~$\BMs \leq \BMb \leq \BM$. We claim that this chain of inequalities often turns into a chain of equalities.
\begin{Le}\label{SimpleMartingale}
Let~$\omega \subset \R^d$ be a strongly martingale connected domain. Let~$f\colon \FixedBoundary\omega\to \R$ be any measurable function. Then\textup,~$\BG_{w,f}(x) = \BMs_{w,f}(x)$ for any~$x\in \omega$.
\end{Le}
\begin{proof}
The proof is a simplification of the proof of Theorem~$2.22$ in~\cite{StolyarovZatitskiy2016}; we provide a comment about simplifications. It suffices to prove the inequalities~$\BG(x) \leq \BMs(x)$ and~$\BMs(x) \leq \BG(x)$ for any~$x\in \omega$. 

To show the first inequality, it suffices to check the inclusion~$\BMs \in \LC{\omega}{f}$ and use the definition of~$\BG$. It is clear that~$\BMs$ satisfies the boundary condition. By the requirement that~$\omega$ is strongly martingale connected,~$\BMs$ does not attain the value~$-\infty$. The local concavity may be verified similar to the proof of Lemma~$2.17$ in~\cite{StolyarovZatitskiy2016}.

To show the reverse inequality~$\BMs(x) \leq \BG(x)$, it suffices to prove that for any simple~$\omega$-martingale~$M$ with~$M_0 = x$ and any~$G\in \LC{\omega}{f}$, the inequality
\eq{
G(M_0) \geq \E G(M_\infty) \geq \E f(M_{\infty})
}
holds true. This inequality follows from Lemma~$2.10$ in~\cite{StolyarovZatitskiy2016} that says that the quantity~$\E G(M_n)$ is non-increasing in this case; note that the assumption that~$M$ is simple cancels the need for the limit argument (compare with the proof of Lemma~2.19 in~\cite{StolyarovZatitskiy2016}).
\end{proof}
A simple modification of the proof above allows to prove a version of Theorem~\ref{Th53}.
\begin{Th}\label{Th53bis}
Let~$\omega \subset \R^d$ be a strongly martingale connected domain. Let~$f\colon \FixedBoundary\omega\to \R$ be a bounded from below function such that~$\BG_{w,f}$ is lower semi-continuous at every point of the fixed boundary. Then\textup,~$\BG(x) = \BM(x)$ for all~$x\in \omega$. 
\end{Th}
We will also need a technical statement in the spirit of Proposition~$2.7$ in~\cite{StolyarovZatitskiy2016}. The proof is also very similar, so, we omit it.
\begin{St}\label{LowerSemicont}
Let~$\omega \subset \R^d$\textup, let~$x\in \FixedBoundary\omega$. Suppose there exists a closed ball~$B_r(x)$ such that~$B_r(x)\cap \omega$ is a closed strictly convex set. Suppose~$\BG_{\omega,f}$ nowhere equals~$+\infty$. Then\textup,~$\BG$ is lower semi-continuous at the point~$x$ provided~$f$ is.
\end{St} 

Surprisingly, the assertion of Lemma~\ref{SimpleMartingale} becomes false without the assumption~$\omega$ is strongly martingale connected. We use a little bit of complex variable notation to indicate points in the plane. Let~$\omega$ be given by the rule
\begin{equation}\label{Example}
\omega = \{z\in \mathbb{R}^2\mid |z|\leq 1\}\setminus \bigcup_{j=0,1,2}\Set{z\in \mathbb{R}^2}{\big|z-\frac12 e^{\frac{2\pi i j}{3}}\big| < \frac{\sqrt{3}}{4}-\frac{1}{239}}.
\end{equation}
\begin{figure}
\begin{center}
\includegraphics[width=0.4\textwidth]{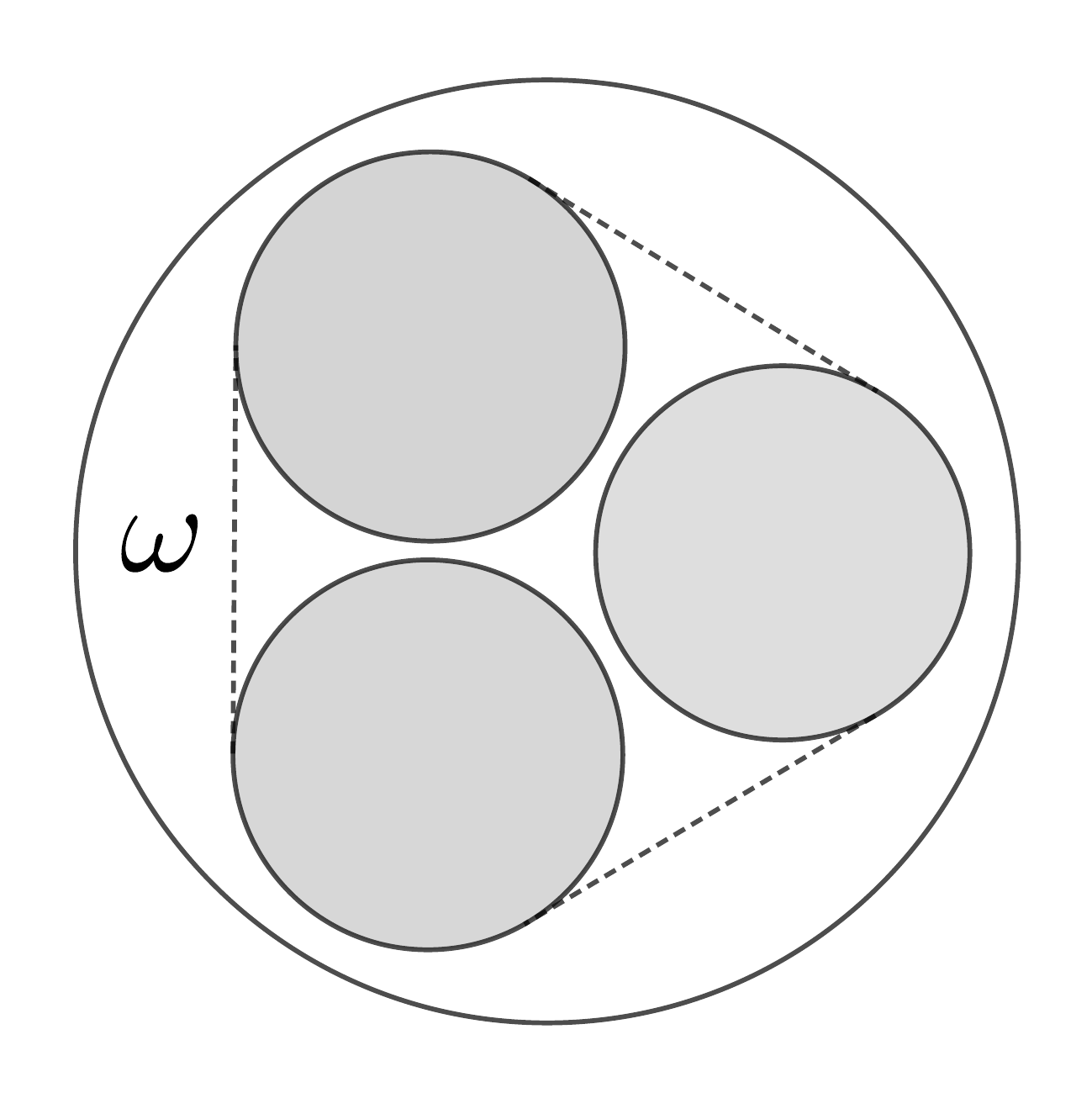}
\end{center}
\caption{An example of a domain that is not strongly martingale connected. The dotted lines mark the convex hull.}
\label{ThreeCircles}
\end{figure}
The main feature of the specific numbers in~\eqref{Example} is that the three small circles almost touch (see Fig.~\ref{ThreeCircles}). By definition,~$\FixedBoundary \omega = \{|z|=1\}$. Let~$f \equiv 0$ on the fixed boundary. In this case, 
\begin{equation*}
\BMs_{\omega,f}(x) = \begin{cases}
-\infty, \quad &x\in \omega \cap \conv\bigg(\bigcup_{j=0,1,2}\Set{z\in \mathbb{R}^2}{\big|z-\frac12 e^{\frac{2\pi i j}{3}}\big| < \frac{\sqrt{3}}{4}-\frac{1}{239}}\bigg);\\
0, \quad &\hbox{otherwise}.
\end{cases}
\end{equation*}
On the other hand, Lemma~$\mathrm{C}.5$ in~\cite{StolyarovZatitskiy2016} says~$\BG_{\omega,f} \geq 0$ (the domain~$\omega$ is a cheese domain in the terminology of that paper). Therefore,~$\BG\equiv 0$ and does not coincide with~$\BMs$. The effect of similar nature arises when one defines rank-one or separate convex hulls, see~\cite{KMS2003}.

\section{Functions on the circle and the line}\label{S5}
During the proof, we will need to work with functions defined on the circle~$\T$ of unit length (in other words, its radius equals~$1/(2\pi)$). Let us equip~$\T$ with the natural arc length measure. We will think of functions on~$\T$ as of periodic functions on the line, i.\,e., identify the function~$\varphi \colon \T \to \R^d$ with its periodic version~$\varphi\per\colon \R\to \R^d$, here
\eq{
\varphi\per(t) = \varphi\Big(\frac{1}{2\pi}e^{2\pi i t}\Big),\qquad t \in \R.
}
Consider a version of the class~\eqref{Class} formed from summable~$\R^d$-valued functions on~$\T$:
\eq{\label{ClassC}
\ClassC = \SSet{\varphi \colon \T \to \partial \OmNull}{\begin{aligned}\exists \ \text{open set}\ \hOmOne\ \text{such that}\ \cl \OmOne\subset \hOmOne\ \text{and}\ \cl\hOmOne\subset \OmNull \\ \text{and for any interval}\  \J\subset \R \quad \av{\varphi\per}{\J}\notin \hOmOne\end{aligned}}.
}
Note that in this definition we require the 'boundedness of oscillation' over arcs that may wind around the circle several times. The additional domain~$\hOmOne$ is mostly needed for technical purposes, it seems to be unavoidable in Lemma~\ref{SplittingLemma} below.

\begin{Rem}\label{rem240401}
By Theorem~\ref{OuterApproximation} below\textup, we may assume~$\hOmOne$ in~\eqref{ClassC} is strictly convex. In this case\textup, the domain $\hOm = \cl \OmNull \setminus \hOmOne$ satisfies~\eqref{StrictConvexity} and~\eqref{ConeCondition}.
\end{Rem}
 Recall~$\Omega^* = \cl\OmNull\setminus \cl\OmOne$.  
\begin{Le}
Assume the domains satisfy~\eqref{StrictConvexity}. In the case~$d=2$ we additionally assume~\eqref{ConeCondition}. For any~$x \in \OmStar$ there exists~$\varphi \in \ClassC$ such that~$\av{\varphi}{\T} = x$.
\end{Le}
\begin{proof}
Construct a segment~$\ell_x$ with the help of Remark~\ref{DomainRemark}: there exist points~$a,b\in \partial \OmNull$ and non-negative numbers~$\alpha$ and~$\beta$ with sum one such that~$x = \alpha a + \beta b$
and~$\ell_x = [a,b] \subset \OmStar$. Take any open $\hOmOne$ such that $\cl \OmOne\subset \hOmOne$ and $\cl\hOmOne\subset \OmNull \setminus \ell_x$. Construct the function~$\varphi$ as in~\eqref{StepFunction} and extend it periodically to the entire line. Then, for any~$\J \subset \R$ the point~$\av{\varphi_{\per}}{\J}$ lies on~$\ell_x$  and therefore not in $\hOmOne$. 
\end{proof}
Consider the Bellman functions
\eq{\label{BellmanC}
\BellC_{\Omega,f}(x) = \sup\Set{\av{f(\varphi)}{\I}}{\varphi \in \ClassC,\ \av{\varphi}{\T} = x},\quad x\in \OmStar,
}
and
\eq{
\BellCb_{\Omega,f}(x) = \sup\Set{\av{f(\varphi)}{\I}}{\varphi \in \ClassC,\ \av{\varphi}{\T} = x, \quad \varphi \in L_{\infty}},\quad x\in \OmStar.
}
Similar to~\eqref{Bellman} and~\eqref{Bellmanb}, we define the function~\eqref{BellmanC} for~$f$ that is bounded from below, whereas the function~$\BellCb$ is a meaningful object for any~$f$ locally bounded from below. Since~$\varphi_{\per}|_{[0,1]} \in \Class$  for any~$\varphi \in \ClassC$, we have
\eq{\label{Monotonicity}
\BellC_{\Omega,f}(x) \leq \Bell_{\Omega,f}(x)\quad \text{and}\quad \BellCb_{\Omega,f}(x) \leq \Bellb_{\Omega,f}(x),\qquad x\in \OmStar.
}
The theorem below is our second main result.
\begin{Th}\label{CircleTheorem}
Let the domains~$\OmNull$ and~$\OmOne$ satisfy the usual requirements~$\cl \OmOne \subset \OmNull$\textup,~\eqref{StrictConvexity}\textup, and~\eqref{ConeCondition}. Let the function~$f$ be lower semi-continuous and bounded from below. Then\textup,
\eq{
\BellC_{\Omega,f}(x) = \BG_{\Omega^*, f}(x),\qquad x\in\Omega^*.
}
\end{Th}

The geometric functions on the right hand sides of Theorems~\ref{IntervalTheorem} and~\ref{CircleTheorem} are closely related.
\begin{St}\label{CoincidenceProposition}
Let the domains~$\OmNull$ and~$\OmOne$ satisfy the requirements~$\cl \OmOne \subset \OmNull$ and~\eqref{StrictConvexity}.  Let the function~$f$ be locally bounded. Then\textup,
\eq{
\BG_{\Omega, f}(x) = \BG_{\Omega^*, f}(x),\qquad x\in\Omega^*.
}
\end{St}
Before we pass to the proofs, we need to survey the theory from~\cite{StolyarovZatitskiy2021}. 
\begin{Def}
We say that two~$Y$-valued random variables~$\zeta_1$ and~$\zeta_2$ are equidistributed provided they have the same distributions\textup, which means
\eq{
P(\zeta_1 \in A) = P(\zeta_2 \in A)
} 
for any measurable set~$A \subset Y$.
\end{Def}
Note that if~$\zeta_1$ and~$\zeta_2$ are equidistributed, then~$\E f(\zeta_1) = \E f(\zeta_2)$ for any function~$f$ such that one of these mathematical expectations exists.
If~$\varphi\in \ClassC$ and~$\J\subset \R$ is an interval, then by~$\mu_{\varphi|_{\J}}$ we denote the distribution of the random variable~$\varphi_{\per}|_{\J}$. In other words,
\eq{
\mu_{\varphi|_{\J}}(A) = \frac{1}{|\J|}\big|\set{t\in \J}{\varphi_{\per}(t) \in A}\big|,\qquad A \ \text{is a Borel subset of}\ \FixedBoundary\Om.
} 
Note that~$\mu_{\varphi|_{\J}}$ is a probability measure on~$\FixedBoundary\Om$. Consider the set~$\M(\FixedBoundary\Om)$ of all probability measures on~$\FixedBoundary\Om$ with bounded first moment. The set~$\M(\FixedBoundary\Om)$ is a convex subset of the space of all finite signed measures on~$\FixedBoundary\Om$  with bounded first moment. Let~$\OMM$ be a subset of~$\M(\FixedBoundary\Om)$. We will always impose two conditions on this set:
\eq{\label{TwoConditions}\begin{aligned}
&1) \text{ The set }\OMM\text{ contains all delta measures, i.\,e., }
\DD = \set{\delta_x}{x\in \FixedBoundary\Om} \subset \OMM; \\
&2) \text{ For any finite collection }x_1,x_2,\ldots, x_N \in \FixedBoundary\Om\text{ the intersection of } \OMM\\
&\quad\text{ with the linear space generated by }\delta_{x_1},\delta_{x_2},\ldots,\delta_{x_N}\\
&\quad\text{ is open in the topology of the latter space.}
 \end{aligned}
}
Note that~$\DD \subset \FixedBoundary\OMM$. We will be working with simple martingales that attain their values in the space~$\M(\FixedBoundary\Om)$; these martingales are easy to define since a simple martingale attains values in a finite dimensional linear space. We denote by~$\mu_{\varphi}$ the distribution of the function~$\varphi$ itself, i.\,e.,~$\mu_\varphi = \mu_{\varphi|_{[0,1]}}$.
\begin{Th}[Theorem $2.3$ in~\cite{StolyarovZatitskiy2021} with slight modifications]\label{SZGluingTheorem}
Let~$\OMM$ be a subset of~$\M(\FixedBoundary\Om)$ that satisfies conditions~\eqref{TwoConditions}. Let~$\mathbb{M}$ be a simple~$(\OMM,\DD)$-martingale in the sense of Remark~\ref{InfDimRem}. Then\textup, there exists~$\varphi\colon \T \to \FixedBoundary\Om$ such that~$\mu_{\varphi} = \mathbb{M}_0$ and for any interval~$\J \subset \R$ we have~$\mu_{\varphi|_{\J}} \in \OMM$.
\end{Th} 
The theorem above will serve as a tool to construct functions~$\varphi \in \ClassC$ with prescribed distributions. We must say about the difference between the formulations above and in~\cite{StolyarovZatitskiy2021}. We require weaker openness condition on~$\OMM$ here, in~\cite{StolyarovZatitskiy2021} the set~$\OMM$ was open in weak-* topology. One may go through the proof in~\cite{StolyarovZatitskiy2021} and realize that everything happens in the finite dimensional space spanned by the values of~$\mathbb{M}_{\infty}$. We will apply the theorem to the sets
\eq{\label{ExampleOfOMM}
\OMM = \Set{\mu \in \M(\FixedBoundary\Om)}{\int\limits_{\R^d} x\,d\mu(x) \notin \cl\OmOne}.
}
Note that the intersection of~$\OMM$ with any finite dimensional space~$V$ spanned by delta measures is open. By definition, the condition~$\varphi\in \ClassC$ is almost equivalent to the condition that~$\mu_{\varphi|_J}\in \OMM$ for any subinterval~$J\subset \R$ (the word `almost' corresponds to the presence of the set~$\hOmOne$ in~\eqref{ClassC}). 
We will later show that any simple~$\Om$-martingale~$M$ generates a simple martingale~$\mathbb{M}$ with values in~$\M(\FixedBoundary\Om)$ that satisfies the conditions of Theorem~\ref{SZGluingTheorem} with~$\OMM$ given in~\eqref{ExampleOfOMM}. This observation will allow us to construct a function $\varphi \in \ClassC$ that has the same distribution as~$M_{\infty}$ (see Lemma~\ref{GluingLemma} below).

\section{Proof of Theorem~\ref{CircleTheorem}}\label{S6}
The proof will be based on two lemmas below.
\begin{Le}[Splitting lemma]\label{SplittingLemma}
Assume~$\Omega$ satisfies~\eqref{StrictConvexity} and~\eqref{ConeCondition}. Let~$\varphi \in \Class$ and let~$\tOmOne$ be an open set such that~$\cl \tOmOne \subset \OmOne$\textup, let~$\tOm = \OmNull \setminus \tOmOne$.  There exists an~$\tOm$-martingale~$M$ such that~$M_\infty$ is equidistributed with~$\varphi$.
\end{Le}
We will call the domain~$\tOm$ as in lemma above an \emph{extension} of~$\Om$.
\begin{Le}[Gluing lemma]\label{GluingLemma}
Assume~$\Omega$ satisfies~\eqref{StrictConvexity} and~\eqref{ConeCondition}. Let~$M$ be a simple~$\Om^*$-martingale. There exists a function~$\varphi \in \ClassC$ that is equidistributed with~$M_\infty$.
\end{Le}

\begin{proof}[Proof of Theorem~\ref{CircleTheorem}.]
It suffices to prove the inequalities
\alg{
\label{CircleThFirstIneq}\BellC_{\Omega,f}(x) \leq &\BG_{\Omega^*, f}(x),\qquad x\in\Omega^*,\\
\label{CircleThSecondIneq}\BellC_{\Omega,f}(x) \geq &\BG_{\Omega^*, f}(x),\qquad x\in\Omega^*.
}

Without loss of generality assume that $\BG_{\Omega^*, f}$ is finite. Let us first prove~\eqref{CircleThFirstIneq}.  Fix~$x\in \OmStar$. By~\eqref{BellmanC}, for any~$\eps > 0$ there exists~$\psi \in \ClassC$ such that
\eq{
\av{\psi}{\T} = x\qquad \text{and}\qquad \av{f(\psi)}{\T} \geq \BellC(x) - \eps. 
}
Let~$\hOmOne$ be the set that corresponds to~$\psi$ in~\eqref{ClassC}. By Remark~\ref{rem240401} we assume~$\hOmOne$ is strictly convex and $\hOm = \cl\OmNull \setminus \hOmOne$ satisfies~\eqref{StrictConvexity} and~\eqref{ConeCondition}. Then,~$\Om$ is an extension of~$\hOm$. We apply Lemma~\ref{SplittingLemma} to the function~$\psi_{\per}|_{[0,1]}\in\Class(\hOm)$ with~$\Om$ in the role of extension of~$\hOm$ and obtain an~$\Om$-martingale~$M$ such that~$M_{\infty}$ is equidistributed with~$\psi$. Then,
\eq{
\BellC_{\Om,f}(x) \leq \av{f(\psi)}{\T} + \eps  = \E f(M_\infty) + \eps \leq \BG_{\OmStar,f}(x)+\eps,
}
the last inequality follows from Theorem~\ref{Th53bis}, Proposition~\ref{LowerSemicont}, and Lemma~\ref{Le52}. It remains to choose arbitrarily small~$\eps$.

Let us now prove~\eqref{CircleThSecondIneq}. By Lemmas~\ref{Le52} and~\ref{SimpleMartingale}, for any $x \in \Omega^*$ there exists a simple~$\OmStar$-martingale~$M$ such that
\eq{
\E f(M_\infty) \geq \BG_{\Omega^*, f}(x) -\eps,\qquad M_0 = x.
}
We apply Lemma~\ref{GluingLemma} and obtain a function~$\varphi \in \ClassC$ equidistributed with $M_{\infty}$. Then, 
\eq{
\av{\varphi}{\T} = \E M_{\infty} = x, \qquad \av{f(\varphi)}{\T} = \E f(M_{\infty}) \geq \BG_{\Omega^*,f}(x) - \eps.
}
It remains to choose arbitrarily small~$\eps$ to prove~\eqref{CircleThSecondIneq}.
\end{proof}

The proof of Lemma~\ref{SplittingLemma} follows the lines of the proof of Theorem~$3.7$ in~\cite{StolyarovZatitskiy2016}. We introduce the function~$\Delta\colon \Omega\setminus \FixedBoundary\Omega \to \R$:
\eq{
\Delta(x) = \sup\Set{\max\Big(1,\frac{|x-y|}{|x-z|}\Big)}{ x\in [y,z],\ y\in\Omega,\ z\in \cl\tOmOne}.
}
We will also often use the notion of a \emph{transversal} segment. See Fig.~\ref{TransSegm} for visualization.
\begin{figure}\label{TransSegm}
\begin{center}
    \includegraphics[width=0.6\textwidth]{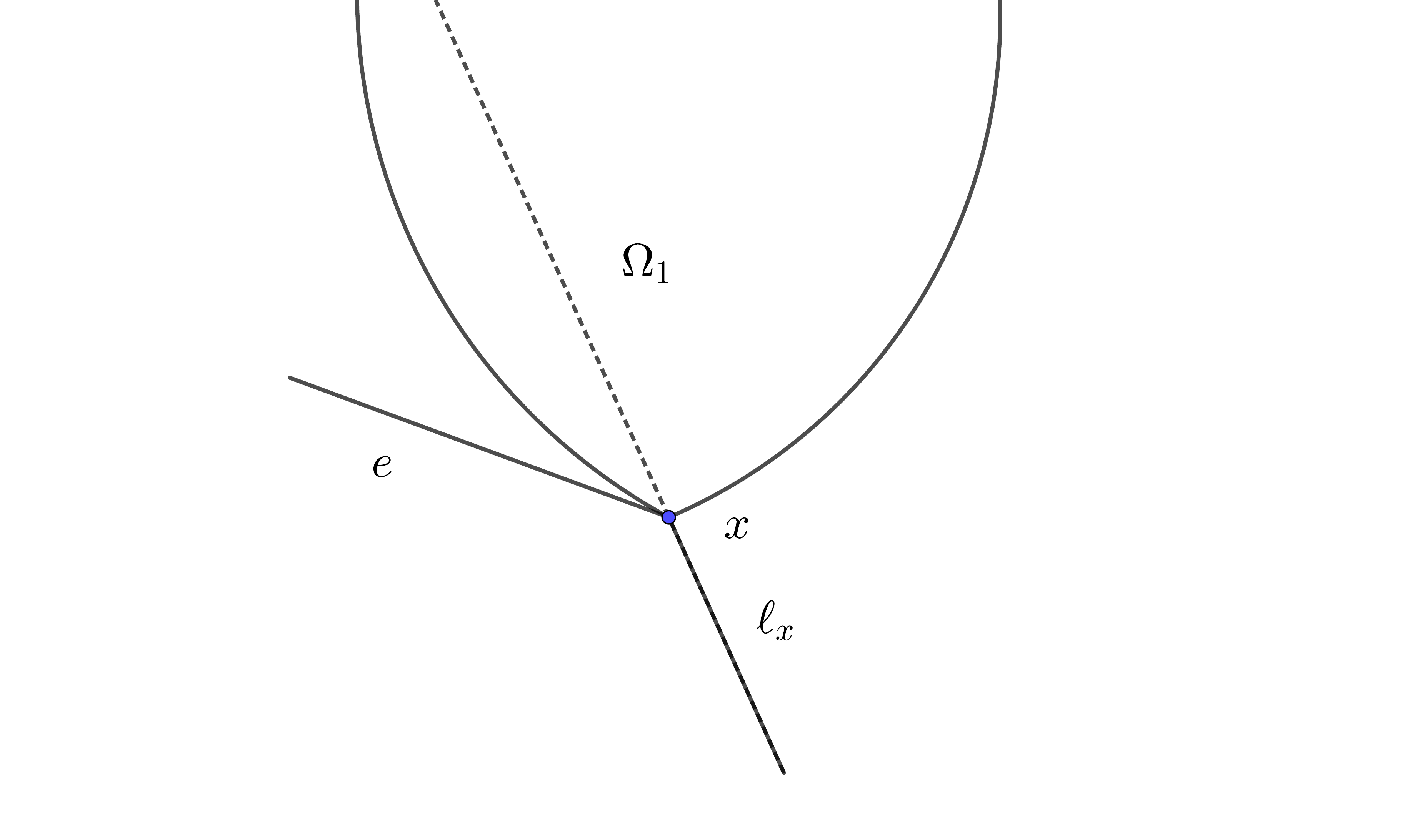}
     \caption{The segment~$\ell_x$ is transversal, the segment~$e$ is not transversal in this case.}
\end{center}		
\end{figure}
\begin{Def}\label{TransversalDef}
Let~$x \in \FreeBoundary \Om$\textup, let~$\ell\subset \Om$ be a segment with the endpoint~$x$. We say that~$x$ is transversal provided the line containing it intersects~$\OmOne$. 
\end{Def}
\begin{Le}
Let~$\OmNull$ and~$\OmOne$ satisfy~\eqref{StrictConvexity}. Let~$\tOmOne\subset \OmOne$ be such that~$\cl\tOmOne\subset \OmOne$. Then\textup, the condition~\eqref{ConeCondition} is equivalent to the fact that for any choice of~$\tOmOne$ the function~$\Delta$ is uniformly bounded on any compact subset of~$\Om$.
\end{Le}
\begin{proof}
Let us first verify the necessity of~\eqref{ConeCondition}. Assume it does not hold and there exists a ray~$L \subset \OmNull$ such that it cannot be shifted inside~$\OmOne$. Without loss of generality, we may assume~$L$ starts from $x\in \FreeBoundary\Om$ and does not intersect~$\OmOne$. By~\eqref{StrictConvexity}, we may also assume~$L$ is a transversal segment. Let us choose~$\tOmOne$ in such a manner that it intersects the continuation of~$L$ beyond~$x$, let~$z$ belong to that intersection. Choosing~$y$ as far as we wish on~$L$, we obtain that the ratio~$|x-y|/|x-z|$, and thus, the value~$\Delta(x)$, is unbounded.

Now we turn to the sufficiency of~\eqref{ConeCondition}. Let~$C\subset \Om$ be a compact set. First, we note that
\eq{
|x-z| \geq \dist(C,\tOmOne) > 0,\qquad x\in C, \ z\in \tOmOne.
}
Second, it suffices to prove that the quantity~$|y-x|$ is uniformly bounded whenever~$x\in C$,~$y\in \Om$, and there exists~$z\in \cl\tOmOne$ such that~$x\in [y,z]$. 
Assume the contrary. Let there exist sequences~$\{x_n\}_{n}$, ~$\{y_n\}_{n}$, and~$\{z_n\}_{n}$ such that
\eq{
x_n\in C,\ y_n\in\Om,\ z_n \in \cl\tOmOne,\quad x_n \in [y_n,z_n],\quad \text{and}\quad |y_n - x_n| \to \infty.
}
Without loss of generality, we may assume that~$x_n \to x$ and~$(y_n - x_n)/|y_n-x_n|\to e$, where~$|e|=1$. By the closedness of~$\Om$, the ray~$L = x+e\cdot \R_+$ lies inside~$\Om$ entirely. By~\eqref{StrictConvexity},~$\Om$ does not contain lines, so~$|z_n|$ is uniformly bounded. We may assume~$z_n \to z \in \cl\tOmOne$. This means~$L \subset \Om$ cannot be shifted to lie inside~$\OmOne$, which contradicts~\eqref{ConeCondition}.
\end{proof}

\begin{proof}[Proof of Lemma~\ref{SplittingLemma}.]
Given a function~$\varphi \in \Class$, there exists a partition~$\I = \I_{1}\cup\I_{2}$, with~$\I_1, \I_2$ being disjoint (up to a common point) intervals such that
\eq{
\big[\av{\varphi}{\I_1},\av{\varphi}{\I_2}\big] \cap\tOmOne = \varnothing\quad \text{and}\quad \max\Big(\frac{|\I_1|}{|\I_2|}, \frac{|\I_2|}{|\I_1|}\Big) \leq \Delta(\av{\varphi}{\I}).
} 
The proof of this statement is completely similar to the proof of Lemma~$3.9$ in~\cite{StolyarovZatitskiy2016}. We apply it inductively to build a sequence~$\{\{\I_k^n\}_{k=1}^{2^n}\}_n$ of partitions of~$\I$ such that
\begin{enumerate}[1)]
\item for each~$n$ the partition~$\{\I_k^{n+1}\}_k$ is a subpartition of~$\{\I_k^n\}_k$, moreover, for each~$n$ and~$k$,~ $1 \leq k \leq 2^n$, one has~$\I^{n+1}_{2k-1} \cup \I^{n+1}_{2k} = \I^{n}_k$;
\item for each~$n$ and~$k$,~ $1 \leq k \leq 2^{n}$, the segment~$\Big[\av{\varphi}{\I_{2k-1}^{n+1}},\av{\varphi}{\I^{n+1}_{2k}}\Big]$ lies in~$\tilde{\Omega}$;
\item for each~$n$ and~$k$,~ $1 \leq k \leq 2^n$,~$\max\Big(\frac{|\I^{n+1}_{2k-1}|}{|\I^{n+1}_{2k}|},\frac{|\I^{n+1}_{2k}|}{|\I^{n+1}_{2k-1}|}\Big) \leq \Delta(\av{\varphi}{\I^n_k})$.
\end{enumerate}
Let~$\F_n$ be generated by~$\{\I_k^n\}_{k=1}^{2^n}$, let~$M_n = \E(\varphi\mid \F_n)$. The martingale~$M = \{M_n\}_n$ is the desired~$\tilde{\Omega}$-martingale (the proof of this assertion is identical to the proof of Theorem~$3.7$ in~\cite{StolyarovZatitskiy2016}). 
\end{proof}
\begin{Rem}
The assertion of Lemma~\ref{SplittingLemma} remains true if~$\Omega$ does not necessarily satisfy~\eqref{ConeCondition}\textup, however\textup,~$\varphi \in L_{\infty}$. The proof should be modified as follows. We choose a compact convex set~$C \subset \R^d$ such that~$\varphi\in C$ almost surely. We consider the function
\eq{
\Delta_C(x) = \sup\Set{\max\Big(1,\frac{|x-y|}{|x-z|}\Big)}{ x\in [y,z],\ y\in\Omega\cap C,\ z\in \cl\tOmOne \cap C},\quad x\in C\cap \Om.
}
This function is bounded since~$|x-z|$ is separated away from zero and $|x-y|$ is bounded. Now we may repeat the proof of Lemma~\ref{SplittingLemma} with~$\Delta_C$ in the role of~$\Delta$. 
\end{Rem}
\begin{proof}[Proof of Lemma~\ref{GluingLemma}.] Let~$M$ be a simple~$\OmStar$ martingale. Note that we may choose the sets~$C_a$ in Definition~\ref{OmegaMart} to be closed simplices. Let~$C$ be the union of such simplices over all atoms of all algebras~$\F_n$. In fact, this is union of a finite number of simplices. Thus,~$C$ is a compact subset of~$\OmStar$. Therefore,~$C$ is separated from~$\OmOne$ and does not intersect with the sets~$\Om_\eps$ for~$\eps$ sufficiently close to~$1$, here
\eq{
\Om_\eps = (1-\eps)\OmNull + \eps\OmOne;
}
we use the standard Minkowski addition. Fix some~$\eps$ close to~$1$ such that~$C\cap \Om_{\eps} = \varnothing$ and set~$\hOmOne = \Om_\eps$. Note that the corresponding lens~$\hOm = \cl\OmNull\setminus \hOmOne$ satisfies~\eqref{StrictConvexity} and~\eqref{ConeCondition}\footnote{Alternatively, the set~$\hOmOne$ may be constructed with the help of Theorem~\ref{OuterApproximation} below.}. What is more,~$\cl\OmOne \subset \hOmOne$. Thus,~$\hOmOne$ fits into formula~\eqref{ClassC} and~$M$ is an~$\hOm$-martingale. Consider the set
\eq{\label{OurOMM}
\OMM = \Set{\mu \in \M(\FixedBoundary \Om)}{\int\limits_{\R^d} x\,d\mu(x) \notin \cl\hOmOne}.
}
This set satisfies the two requirements on the set~$\OMM$ listed in~\eqref{TwoConditions}. 
It is high time to make our choice for the martingale~$\mathbb{M}$. Recall the martingale~$M$ is simple. The desired martingale is defined by the formula
\eq{
\mathbb{M}_n(w) = \mu_{M_{\infty}|_w},\qquad w\ \text{is an atom of}\ \F_n,
}
here~$\mu_{\zeta}$ denotes the distribution of the random variable~$\zeta$; we treat~$M_{\infty}|_w$ as a random variable on the probability space~$w$ equipped with the measure~$(P(w))^{-1}P|_w$. Let us prove that~$\mathbb{M}$ is an~$(\OMM,\DD)$-martingale with~$\OMM$ given in~\eqref{OurOMM}. We will firstly show that~$\mathbb{M}$ is indeed a martingale. For that we choose an arbitrary~$n$ and an atom~$w\in \F_n$. Let~$w_1,w_2,\ldots,w_j \in \F_{n+1}$ be the kids of~$w$. The martingale property of~$\mathbb{M}$ is
\eq{
P(w)\mu_{M_\infty|w} = \sum\limits_j P(w_j)\mu_{M_{\infty}|_{w_j}}.
}
To prove this identity in measures, we may test it against a Borel set~$A \subset \R^d$:
\eq{
P(w)P(M_{\infty}|_w\in A) = \sum\limits_j P(w_j)P(M_{\infty}|_{w_j}\in A),
}
which is true. It remains to verify the third property in Definition~\ref{OmegaMart}. We need to check that any convex combination~$\sum_{j}\alpha_j \mu_{M_\infty|_{w_j}}$ lies inside the set~$\OMM$. This means
\eq{
\sum\limits_j\int\limits_{\partial\OmNull} \alpha_j x\,d\mu_{M_{\infty}|_{w_j}}(x) \notin \cl\hOmOne.
}
This reduces to the fact that~$M$ is an~$\hOm$-martingale since
\eq{
\int\limits_{\partial\OmNull}x\,d\mu_{M_{\infty}|_{w_j}}(x) = M_{n+1}(w_j).
}
We apply Theorem~\ref{SZGluingTheorem} to~$\mathbb{M}$ and~$\OMM$ and obtain a function~$\varphi \in \ClassC$ with~$\mu_{\varphi} = \mathbb{M}_0$. It remains to notice that~$\mathbb{M}_0$ is the distribution of~$M_\infty$.
\end{proof}

\section{Proof of Proposition~\ref{CoincidenceProposition}}\label{S7}
First, the inequality
\eq{\label{eq71}
\BG_{\Omega, f}(x) \geq \BG_{\Omega^*, f}(x),\qquad x\in\Omega^*,
}
follows from~\eqref{DefinitionOfMinimalLocallyConcave} since whenever~$G\in \LC{\Om}{f}$, its restriction~$G|_{\OmStar}$ to~$\OmStar$ belongs to~$\LC{\OmStar}{f}$. Second, to prove the reverse inequality to~\eqref{eq71}, it suffices to construct a function~$G\in \LC{\Om}{f}$ such that
\eq{
G(x) = \BG_{\OmStar,f}(x),\qquad x\in \OmStar.
}
The construction of the function~$G$ is fairly straightforward, however, the verification of its local concavity will take some time. To construct~$G$, we will use special segments~$\ell \subset\Om$. For any~$x\in\FreeBoundary\Om$ let us choose some transversal (see Def.~\ref{TransversalDef}) segment~$\ell_x$ with the endpoint~$x$. Define the function~$G$ by the formula
\eq{\label{Continuation}
G(x) = 
\begin{cases}
\BG_{\OmStar,f}(x),\qquad &x\in \OmStar;\\
\lim\limits_{\genfrac{}{}{0pt}{-2}{y\to x}{y\in \ell_x}} \BG_{\OmStar,f}(y),\qquad &x\in\FreeBoundary\Om,
\end{cases}
}
where~$y$ approaches~$x$ along~$\ell_x$. Note that the limit in the formula always exists (though it might be equal to~$-\infty$) since~$\BG_{\OmStar,f}|_{\ell_x}$ is a concave function. 

\begin{Le}\label{TransLemma}
Let~$\Omega$ satisfy~\eqref{StrictConvexity}\textup, let~$x\in \FreeBoundary\Omega$. Let~$\ell$ be a transversal segment with the endpoint~$x$. Let~$s \subset \Om$ be another segment with the endpoint~$x$. Then\textup, the convex hull of~$\ell$ and~$s$ also belongs to~$\Om$ entirely.
\end{Le}
\begin{proof}
It suffices to prove that the said convex hull is disjoint with~$\OmOne$. Assume the contrary, let~$y\in \OmOne$ lie in the convex hull of~$\ell$ and~$s$. Let~$z$ be a point on the continuation of~$\ell$ over the point~$x$ that is sufficiently close to~$x$. Since~$\ell$ is a transversal segment,~$z\in \OmOne$. The segment~$[z,y]$ then lies inside~$\OmOne$ since~$\OmOne$ is convex. On the other hand,~$[z,y]$, clearly, intersects~$s$, which contradicts~$s\subset \Om$.
\end{proof}
\begin{Rem}\label{TransRem}
In fact\textup, the said convex hull lies inside~$\OmStar,$ except for the point~$x$ itself.
\end{Rem}
\begin{proof}[Proof of Proposition~\ref{CoincidenceProposition}]
As we have said, it suffices to show~$G$ given in~\eqref{Continuation} is locally concave on~$\Om$ (in particular, we need to verify that~$G$ does not attain the value~$-\infty$). The verification of local concavity consists of checking the inequalities
\eq{\label{DesiredConvexity}
G(x) \geq\alpha G(a) + \beta G(b),\qquad x = \alpha a + \beta b,\ \alpha + \beta = 1,\ \alpha,\beta >0,  \quad [a,b]\subset \Om.
}
We are interested in the cases where one of the points~$x, a$, or~$b$ lies on~$\FreeBoundary \Omega$. Let~$a,b,x$ be distinct points.

\paragraph{Case~$x \in \FreeBoundary\Om$.} Note that in this case~$a$ and~$b$ do not lie on~$\FreeBoundary\Om$ by~\eqref{StrictConvexity}. Thus, we may assume they are interior points of~$\Om$. Consider the segment~$\ell_x$, some point~$y \in \ell_x$ (let~$y\ne x$), and the points~$a_\gamma$,~$b_\gamma$, and~$x_\gamma$ given by the rule
\eq{\label{GammaDefinition}
z_\gamma = \gamma y + (1-\gamma) z,\quad \gamma \in (0,1],
}
here~$z$ stands either for~$a$, or for~$b$, or for~$x$. Note that~$[a_{\gamma},b_{\gamma}] \subset \OmStar$ by Lemma~\ref{TransLemma} (with Remark~\ref{TransRem}). Thus,
\eq{\label{PreDesiredConvexity}
\BG_{\OmStar,f} (x_\gamma) \geq \alpha \BG_{\OmStar,f}(a_\gamma) + \beta \BG_{\OmStar,f}(b_\gamma).
}
Note that~$z_\gamma \to z$ as~$\gamma \to 0$. Thus,~\eqref{DesiredConvexity} is proved in this case since~$\BG_{\OmStar,f}$ is continuous at~$a$ and~$b$. 

The reasoning above also shows that~$G(x) > -\infty$. Indeed, we need to choose some points~$a,b\in \OmStar$ such that~$x\in [a,b]$ and use~\eqref{DesiredConvexity}.

\paragraph{Case~$a\in \FreeBoundary\Om$.} Consider the segment~$\ell_a$. Let~$y \in \ell_a$. We consider the points~$a_\gamma$, $b_\gamma$, and~$x_\gamma$ defined by the same formula~\eqref{GammaDefinition}. By Lemma~\ref{TransLemma} (Remark~\ref{TransRem}), these points lie inside~$\OmStar$ together with the segment~$[a_\gamma, b_\gamma]$. Thus,~\eqref{PreDesiredConvexity} holds true and~\eqref{DesiredConvexity} follows by a limit argument.
\end{proof}
\begin{Rem}
One may prove that the definition of the function~$G$ by~\eqref{Continuation} does not depend on the particular choice of transversal segments~$\ell_x,$~$x\in \FreeBoundary\Om$.
\end{Rem}

\section{Proof of Theorem~\ref{IntervalTheorem}}\label{S8}
Theorem~\ref{IntervalTheorem} immediately follows from~\eqref{Monotonicity}, Theorem~\ref{CircleTheorem}, and Proposition~\ref{CoincidenceProposition} once we prove the 'extension' theorem below.
\begin{Th}\label{ExtensionGen}
Let~$\Omega$ be a lens that satisfies~\eqref{StrictConvexity} and~\eqref{ConeCondition}. Assume~$\FreeBoundary\Omega$ is~$C^2$-smooth and~$f$ is an arbitrary function. Then\textup,
\eq{\label{ExtensionIdentity}
\BG_{\Om,f}(x) = \inf\Set{\BG_{\tOm,f}(x)}{\,\tOm\ \text{is an extension of}\ \Om}, \quad x\in \Om.
}
\end{Th}
When proving Theorem~\ref{ExtensionGen}, we may assume, without loss of generality, that~$\BG_{\Omega,f}$ is finite. Then, the identity~\eqref{ExtensionIdentity} may be reformulated: for any~$x\in\Om$ and any~$\eps > 0$ there exists an extension~$\tOm$ such that 
\eq{
\BG_{\Om,f}(x) \leq \BG_{\tOm,f}(x) \leq \BG_{\Om,f}(x) + \eps.
}
Note that~$\tOm$ may depend on~$\eps$ and~$x$. In the case of smooth boundaries and sufficiently regular~$f$, we will prove a stronger statement, which is the main step toward the proof of Theorem~\ref{ExtensionGen}.
\begin{Th}\label{Extension}
Let~$\Omega$ be a lens that satisfies~\eqref{StrictConvexity} and~\eqref{ConeCondition}. Assume the boundaries~$\FixedBoundary\Omega,$~$\FreeBoundary\Omega$ are~$C^2$-smooth and the function~$f$ is~$C^2$-smooth as well. Then\textup, for any~$\eps > 0$ there exists an extension~$\tOm$ such that
\eq{
\BG_{\Om,f}(x) \leq \BG_{\tOm,f}(x) \leq \BG_{\Om,f}(x) + \eps
}
for any~$x\in \Om$.
\end{Th}
The proof of this theorem follows the proof of Theorem~$4.1$ in~\cite{StolyarovZatitskiy2016} with some modifications. These modifications do not require new ideas, however, some re-phrasing is needed to work in~$\R^d$ with arbitrary~$d$ instead of~$\R^2$. The condition~$f\in C^2$ may be immediately replaced with~$f$ being merely continuous by approximation in the uniform norm. Here we have used the obvious inequalities
$$
\BG_{\Om, f} \leq \BG_{\Om, g} \leq \BG_{\Om, f + \eps} = \BG_{\Om, f} + \eps,
$$
provided $f \leq g\leq f +\eps$.

\begin{Cor}
The statement of Theorem~\ref{Extension} holds true if~$f\in C(\FixedBoundary\Omega)$.
\end{Cor}
The method of~\cite{StolyarovZatitskiy2016} was to perturb the function~$\BG_{\Om,f}$ a little bit to make it strongly concave and then to extend it through the free boundary. The reasoning naturally splits into two steps: first, we study the boundary behavior of minimal locally concave functions and, second, use this structure to construct the extension. 

\subsection{Boundary behavior of minimal locally concave functions}
\begin{Def}
Let~$\omega \subset \R^d$. We say that two points~$x,y\in\omega$ see each other if~$[x,y]\subset \omega$. The set
\eq{
\Vis_x^\om = \set{y\in \omega}{x\ \text{and}\ y\ \text{see each other in}\ \omega}
}
is called the set of points visible from~$x$.
\end{Def}
We will simply write~$\Vis_x$ instead of~$\Vis_x^\om$ when the ambient set~$\om$ is clear from the context.
\begin{St}\label{BoundedVisibility}
Let~$\Omega$ be a lens that satisfies~\eqref{StrictConvexity} and~\eqref{ConeCondition}. The set~$\Vis_x$ is compact whenever~$x\in\FreeBoundary\Omega$. The diameter of~$\Vis_x$ is uniformly bounded when~$x$ runs through a compact subset of~$\FreeBoundary \Om$.
\end{St}
\begin{proof}
It is clear that the set~$\Vis_x$ is closed (since~$\Omega$ is closed). So we only need to prove the second assertion. Assume the contrary, let there exist a sequence~$\{x_n\}_n$ with values in a compact subset of~$\FreeBoundary\Omega$ and a sequence~$\{y_n\}$ such that~$|y_n| \to \infty$ and~$[x_n,y_n] \subset \Om$. Without loss of generality,~$x_n \to x\in\FreeBoundary \Omega$ and~$y_n/|y_n|\to y\in S^{d-1}$. Then, by the closedness of~$\Om$, the ray~$x+\R_{+}y$ lies inside~$\Om$. By~\eqref{ConeCondition}, there exists~$z\in\OmOne$ such that~$z+\R_+y\subset \OmOne$. This contradicts the strict convexity of~$\OmOne$ since~$x\in \cl\OmOne$.
\end{proof}
\begin{Rem}
The set~$\Vis_x$ is not necessarily bounded if~$x\notin\FreeBoundary\Omega$ as the left picture on Figure~\ref{RayCondition} shows.
\end{Rem}
We will often use the following simple principle (compare with Fact~$B.3$ in~\cite{StolyarovZatitskiy2016}).

\begin{Fact}\label{Slicing}
Let~$\omega$ be a strictly convex subset of~$\R^d$. Suppose that there are positive $r$ and $R$ and a point $p \in \omega\cap L$ such that $B_r(p)\subset \omega$\textup, $B_R(p) \supset \omega \cap L$. Then for $C = \frac{r+R}{r}$ and any $y \in \partial \omega$ one may find $z \in \partial \omega \cap L$ such that $|y-z| \leq C \dist(y,L)$.
\end{Fact}
For the hint to the proof see~Figure~\ref{Fact_proof}. 
\begin{figure}
 \begin{center}
    \includegraphics[width=0.45\textwidth]{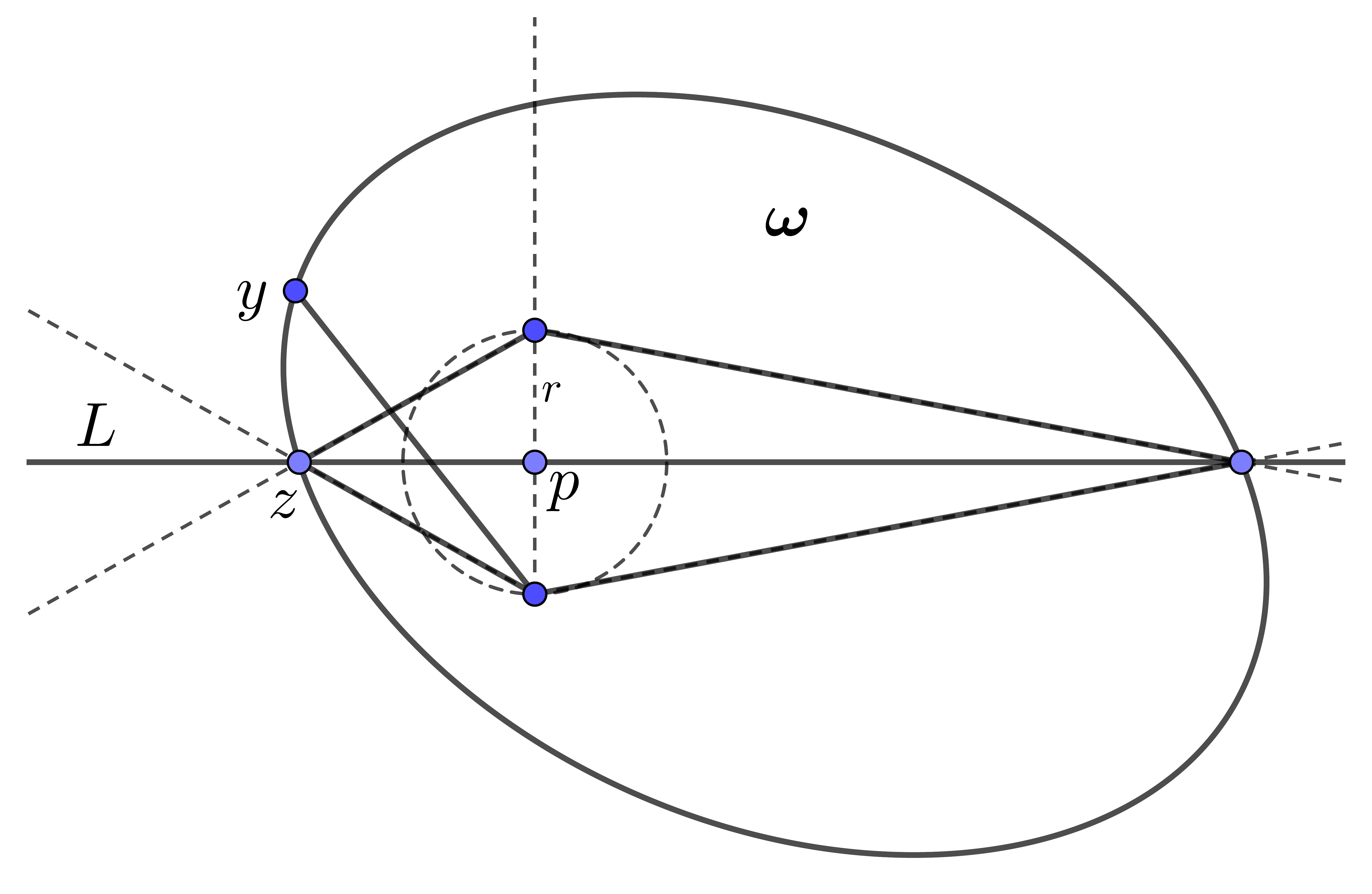}
    \caption{Hint to the proof of Fact~\ref{Slicing}.}
\end{center}		
    \label{Fact_proof}
\end{figure}


\begin{Def}
Let~$x\in \omega \subset \R^d$ and let~$G\colon\omega \to \R$. The set of linear functions $L$ such that 
$G(y)\leq G(x) + L(y-x)$ holds for any~$y\in \Vis_x^\omega$ is called the \emph{superdifferential} of~$G$ at~$x$. We will denote it by~$\eth G|_x$.
\end{Def}

\begin{Fact}\label{Fact_superdiff}
Let~$\omega \subset \R^d$ and let~$G\colon\omega \to \R$. Assume for every~$x\in\omega$ the superdifferential~$\eth G|_x$ of~$G$ at~$x$ is non-empty. 
 Then\textup, the function~$G$ is locally concave on~$\omega$.
\end{Fact}

Every concave function has a non-empty superdifferetial at interior points of its domain (see, e.\,g., Section $23$ in~\cite{Rockafellar1970}). One may ask whether the superdifferential is non-empty for every~$x$ in~$\omega$ provided~$G$ is locally concave. The answer to this question is negative in general (consider the domain~$\omega$ formed by three lines passing through the origin in~$\R^2$). However, for some good domains (lenses among them) and sufficiently good functions, the answer is positive.
\begin{St}\label{ExistenceOfTangent}
Let~$\Omega$ be a lens that satisfies~\eqref{StrictConvexity} and~\eqref{ConeCondition}. Assume that the boundaries of~$\Omega$ are~$C^1$-smooth. Let~$f\colon \FixedBoundary\Omega\to \R$ be a locally Lipschitz function such that~$\BG_{\Omega,f}$ is finite. For any~$x\in\FreeBoundary\Omega$ there exists a linear function~$L[\BG,x]$ such that
\eq{
\BG_{\Omega,f}(x) + L[\BG,x](y-x) \geq \BG_{\Omega,f}(y)
}
for all~$y\in\Vis_x^\Om$. In other words\textup{,} $\eth \BG|_x$ is non-empty.
\end{St}
\begin{Rem}
One may replace the minimal locally concave function~$\BG$ with an arbitrary locally concave function~$G$ provided~$G$ is locally Lipschitz.
\end{Rem}
\begin{proof}
Without loss of generality, we may assume~$x=0$ and
\eq{
T_0\Omega_1 = \{y\in\R^d\mid y_d=0\}.
}
The symbol~$T_z\Omega_1$ denotes the tangent plane to~$\Omega_1$ at the point~$z$. We also assume~$y_d >0$ on~$\Omega_1$. By concavity of~$\BG|_{y_d=0}$, there exists a linear function~$\ell\colon \R^{d-1}\to \R$ such that
\eq{
\BG(z_1,z_2,\ldots,z_{d-1},0) \leq \BG(0) + \ell(z)
} 
for any~$z\in\R^{d-1}$ satisfying $(z,0) \in \Omega$. Let~$\pi$ denote the orthogonal projection onto~$\{y\in\R^d\mid y_d=0\}$. Let
\eq{\label{FirstSup}
a = \sup\Set{\frac{\BG(y) - \BG(0) - \ell(\pi[y])}{-y_d}}{y\in\Vis_0, y_d<0}.
}
We claim that 
\mlt{\label{SecondSup}
a =  \sup\Set{\frac{\BG(y) - \BG(0) - \ell(\pi[y])}{-y_d}}{y\in\Vis_0\cap \FixedBoundary\Om, y_d<0}\\
=\sup\Set{\frac{f(y) - \BG(0) - \ell(\pi[y])}{-y_d}}{y\in\Vis_0\cap \FixedBoundary\Om, y_d<0}.
}
Indeed,~\eqref{SecondSup} clearly does not exceed~\eqref{FirstSup}. On the other hand, if
\eq{
\BG(y) \leq -by_d + \BG(0)+ \ell(\pi[y]) 
}
for all~$y\in \Vis_0\cap \FixedBoundary\Om$ and some~$b\in \R$, then the same inequality holds true for all~$y\in\Vis_0$ by minimality of~$\BG$. Indeed, if this is not the case, the function 
\eq{
y \mapsto \begin{cases}
\BG(y), &y \in \Omega \setminus \Vis_0,\\
\min(\BG(y), -by_d + \BG(0)+ \ell(\pi[y])), & y \in \Omega \cap \Vis_0
\end{cases}
}
lies in $\LC{\Omega}{f}$ and is smaller than $\BG$, which is a contradiction. Thus, the two supremums on the right hand sides of~\eqref{FirstSup} and~\eqref{SecondSup} coincide.

Now let us prove that the local Lipschitz property of~$f$ implies~$a$ is a finite quantity. Let~$\{y^n\}_n$ be a sequence of points in~$\Vis_0\cap \FixedBoundary\Om$ that realizes the supremum in~\eqref{SecondSup}. Assume~$y^n \to y^*$. If~$y^*_d\ne 0$, then $a$ is finite. 

Consider the case~$y_d^*=0$. By Fact~\ref{Slicing}, there exist points~$\tilde{y}^n\in\partial\OmNull$ such that~$\tilde{y}^n_d=0$ and
\eq{
|y^n - \tilde{y}^n|\lesssim y_d^n.
}
Then,
\eq{
f(y^n) - \BG(0) - \ell(\pi[y^n]) = f(\tilde{y}^n) - \BG(0) - \ell(\pi[\tilde{y}^n]) + O(|y_d^n|) \leq O(|y_d^n|),
}
the constant in~$O$ depends on the Lipschitz constant of~$f$ and the constant~$C$ in Fact~\ref{Slicing}. The last relation proves~$a$ is finite. We may set
\eq{\label{LDef}
L[\BG,0](y) = \ell(\pi[y])- ay_d.
}
\end{proof}
\begin{Rem}
It follows from construction that the coefficients of the linear function~$L[\BG,x]$ are uniformly bounded when~$x$ runs through a compact subset of~$\FreeBoundary\Om$.  
\end{Rem}

\begin{St}\label{C2Bound}
Let~$\Omega$ be a lens that satisfies~\eqref{StrictConvexity} and~\eqref{ConeCondition}. Assume that the boundaries of~$\Omega$ are~$C^2$-smooth. Let~$f\colon \FixedBoundary\Omega\to \R$ be a~$C^2$-smooth function such that~$\BG_{\Omega,f}$ is finite. Let~$L[\BG,x]$ be the linear functions constructed by formulas~\eqref{SecondSup} and~\eqref{LDef}\textup, here~$x\in \FreeBoundary \Om$. Then\textup, there exists a point~$e_x\in \FixedBoundary\Om \cap\Vis_x$ such that 
\eq{\label{C2Estimate}
|\BG(x) + L[\BG,x](z-x) - f(z)| = O(|e_x-z|^2),\qquad z\in \FixedBoundary\Om \cap \Vis_x.
}
The constant in this inequality is uniform when~$x$ runs through a compact set on~$\FreeBoundary\Om$.
\end{St}
\begin{proof}
Similar to the proof of Proposition~\ref{ExistenceOfTangent}, we assume~$x=0$,
\eq{
T_0\Omega_1 = \{y\in\R^d\mid y_d=0\},
}
and~$y_d >0$ on~$\Omega_1$.

\newcommand{\AAA}{A}

Let~$e$ be a limit point of the sequence~$\{y^n\}_n$ that realizes the supremum in~\eqref{SecondSup}. We will show that the choice~$e_x:= e$ with~$x=0$ fulfills~\eqref{C2Estimate}. Let us first prove~$\BG(0) + L[\BG,0](e) = f(e)$. If~$e_d\ne 0$, then this follows from the definition of~$L[\BG,0]$. In the case~$e_d=0$, we have~$\BG(e) = \BG(0) + \ell(e_1,e_2,\ldots,e_{d-1})$; if this identity does not hold, the supremum~$a$ equals~$-\infty$, which is definitely false. Therefore,~$\BG(0)+L[\BG,0](e) = f(e)$.

Let us consider the $C^2$-smoth function 
\eq{
\AAA(z) = f(z) - \BG(0) - L[\BG,0](z), \qquad z \in \FixedBoundary\Om.
}

We have $\AAA(e)=0$. If $e_d <0$, then $\AAA$ is non-positive in a neighborhood of $e$. Therefore, $e$ is a local maximum of~$\AAA$, and~\eqref{C2Estimate} follows.

If $e_d=0$, then we know that $\AAA(z) \leq 0$ for $z \in \FixedBoundary\Om$ such that~$z_d \leq 0$. From~\eqref{SecondSup} we know that for any $\delta>0$ we have $\AAA(y^n) > \delta y^n_d$ for sufficiently large $n$. Let us consider the hyperplane that contains the intersection $T_{e,\OmNull}\cap T_{0,\OmOne}$ and is parallel to the $y_d$-axis. Without loss of generality, we may assume this hyperplane is $\{y_{d-1}=0\}$. Let $\tilde\AAA$ be the projection of $\AAA$ onto this plane defined near $e$: if~$\pi\colon \FixedBoundary\Om \to \{y_{d-1} = 0\}$ is the orthogonal projection, then
\eq{
\tilde\AAA(\tilde z) = \AAA(\pi^{-1}[\tilde z]),\qquad \tilde z_{d-1} = 0, \quad \tilde z \text{ is close to } e.
} 
The function $\tilde\AAA$ is $C^2$-smooth in a neighborhood of $e$, $\tilde \AAA(e) = 0$, and $\tilde \AAA(z) \leq 0$ when $\tilde z_d\leq 0$. What is more, there is a sequence $\tilde y^n = (\tilde y^n_1, \dots, \tilde y^n_{d-2}, 0 , \tilde y^n_d)$ such that~$\tilde y^n \to e$, $\tilde y^n_d<0$, and 
\eq{\label{eq120401}
0\geq \tilde\AAA(\tilde y^n) > \delta\tilde y^n_d
}  
for any fixed $\delta>0$ provided $n$ is sufficiently large.

Let us prove that $\nabla \tilde \AAA (e) =0$. The restriction of~$\tilde\AAA$ to the section $y_d=0$ attains its maximum at $e$, therefore, we only need to check that 
\eq{
\frac{\partial \tilde \AAA}{\partial y_d}(e) =0.
} 
Note that the derivative on the left hand side cannot be negative since~$\tilde \AAA(z) \leq 0$ when $\tilde z_d\leq 0$. Let us prove it is non-positive. If this is not the case, then $\frac{\partial}{\partial y_d}\tilde\AAA(e) > 2\delta$ for some $\delta>0$, and, by $C^1$-continuity, it follows that $\frac{\partial}{\partial y_d}\tilde\AAA > \delta$ in a neighborhood of $e$. Then, 
\eq{
\tilde\AAA(\tilde y^n) \leq \tilde\AAA(\tilde y_1^n, \dots, \tilde y_{d-2}^n, 0, 0) + \delta \tilde y^n_d \leq \delta \tilde y^n_d,  
}  
which contradicts~\eqref{eq120401}. Therefore, $\nabla \tilde \AAA (e) =0$. Then, for any $z \in \FixedBoundary\Om$ sufficiently close to $e$ we have 
\eq{
\AAA(z) = \tilde \AAA (\tilde z) = O(|\tilde z - e|^2) = O(|z-e|^2),\qquad \tilde z = \pi [z],
}
which proves~\eqref{C2Estimate}.
\end{proof}
\begin{Rem}
The condition~$f\in C^2$ is superfluous. One may replace it with~$C^{1,1}$ as it can be seen from the proof. 
\end{Rem}
\begin{St}\label{CubicProp}
Let~$\Omega$ be a lens that satisfies~\eqref{StrictConvexity} and~\eqref{ConeCondition}. Assume that the boundaries of~$\Omega$ are~$C^2$-smooth. Let~$f\colon \FixedBoundary\Omega\to \R$ be a~$C^2$-smooth function such that~$\BG_{\Omega,f}$ is finite. Let~$L[\BG,x]$ be the linear functions constructed by formulas~\eqref{SecondSup} and~\eqref{LDef}\textup, here~$x\in \FreeBoundary\Omega$. Then\textup,  
\eq{\label{C2Estimate2}
\BG(y) \leq \BG(x) + L[\BG,x](y-x) + O(|x-y|^3),\qquad x,y\in\FreeBoundary\Omega.
}
The implicit constants hidden by the~$O$-notation are uniform when~$x$ and~$y$ run through a compact set on~$\FreeBoundary\Om$.
\end{St}
\begin{proof}
By compactness argument, it suffices to consider the case where~$y$ lies in a neighborhood of~$x$. Similar to the proof of Proposition~\ref{ExistenceOfTangent}, we assume~$x=0$,
\eq{
T_0\Omega_1 = \{y\in\R^d\mid y_d=0\},
}
and~$y_d >0$ on~$\Omega_1$. Let~$\FreeBoundary\Om$ be defined as the graph of the function~$h\colon U\to \R$, where~$U\subset \R^{d-1}$ is a neighborhood of the origin. Consider the cone
\eq{
C_y = \Set{z\in \R^d}{z_d - y_d \leq -C|y_{\bar d}||z_{\bar d} - y_{\bar d}|}, \quad y=(y_{\bar d},y_d), z=(z_{\bar{d}},z_d),\quad y_{\bar d}, z_{\bar d} \in \R^{d-1},
}
where~$y \in \FreeBoundary\Om$ is a point in a neighborhood of~$x$ and~$C$ is a sufficiently large constant. See Fig.~\ref{TheConeC1} for visualization.
\begin{figure}
    \includegraphics[width=0.45\textwidth]{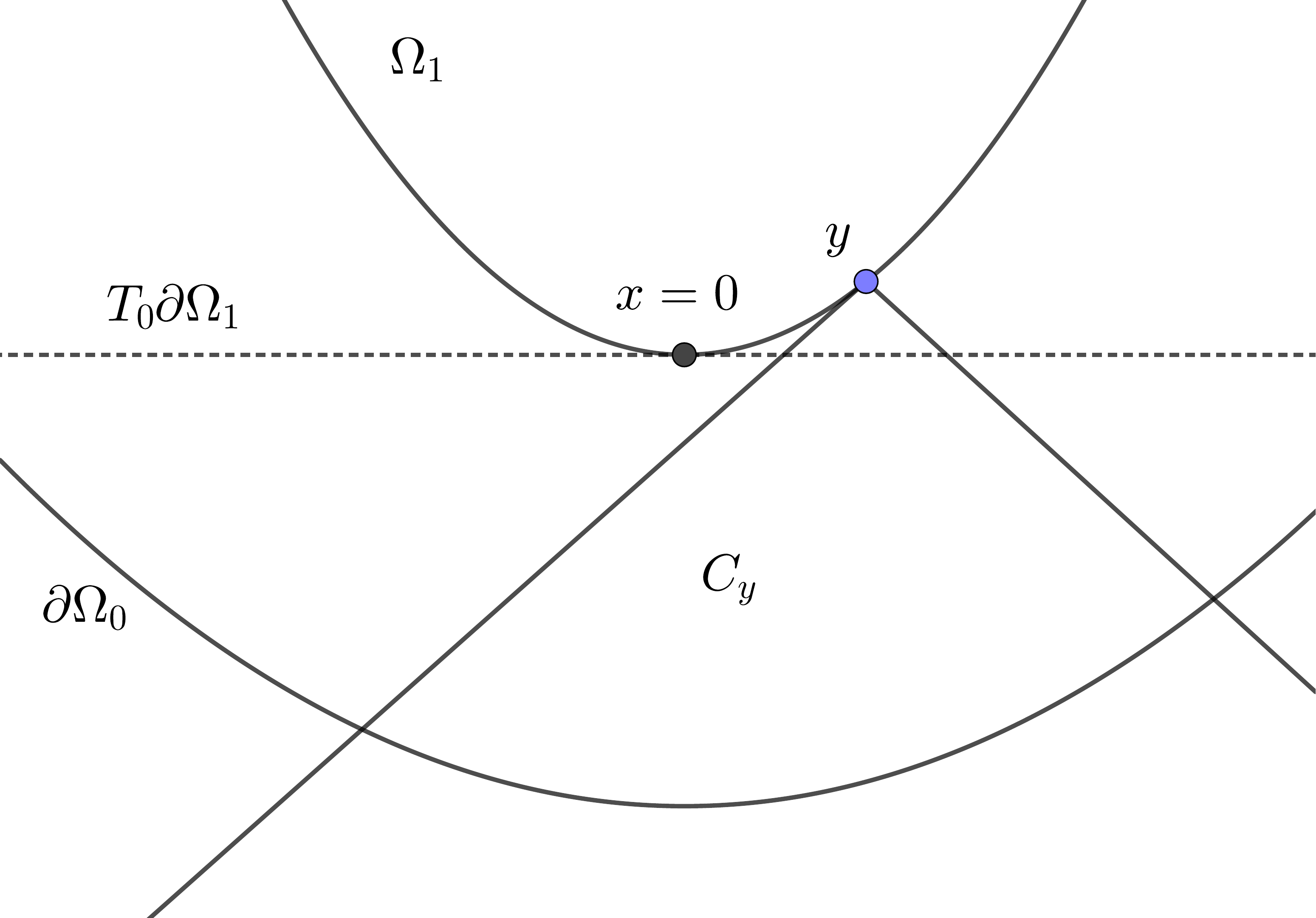}
    \hspace{1cm}
    \includegraphics[width=0.45\textwidth]{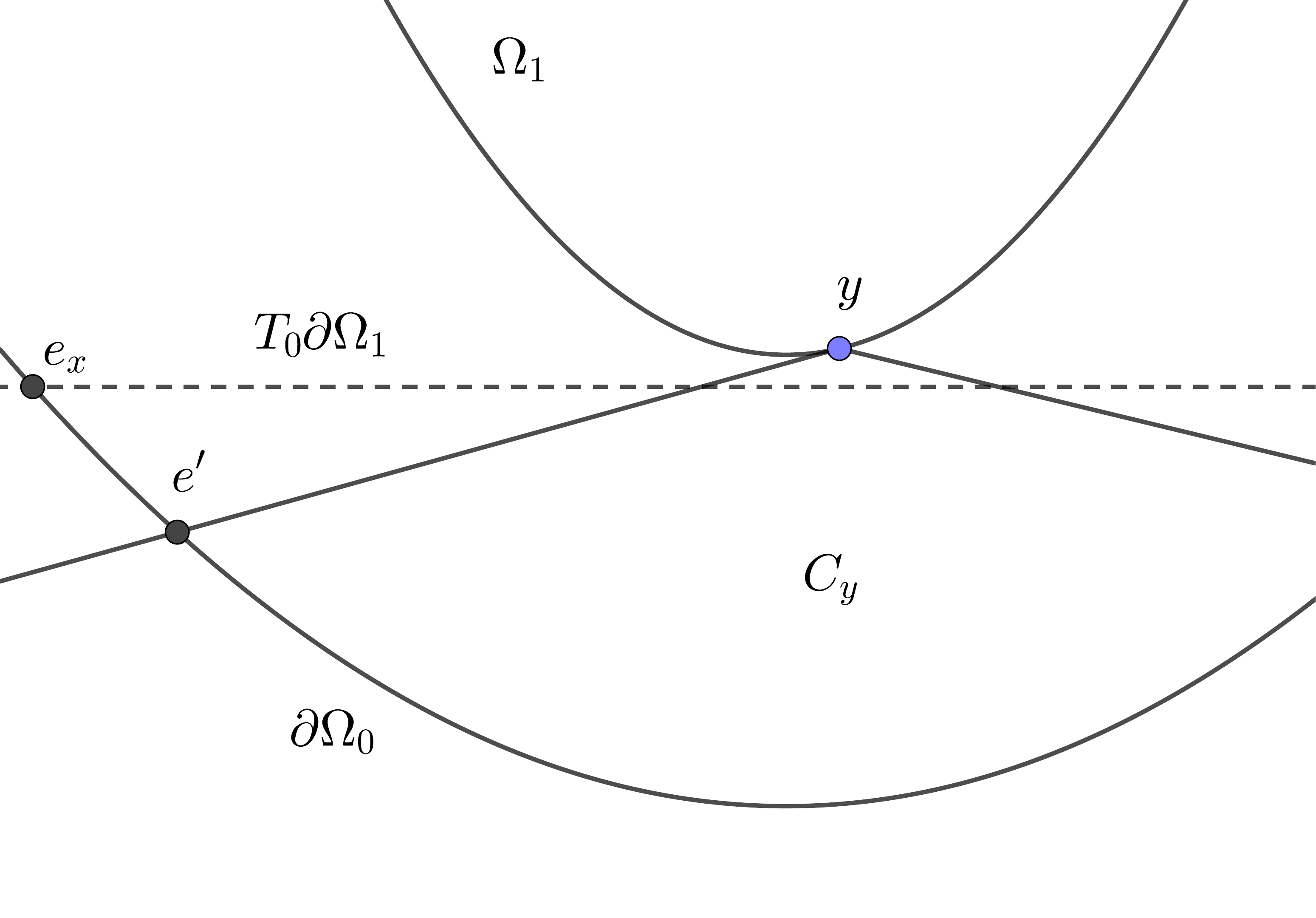}
    \caption{Illustration to the proof of Proposition~\ref{CubicProp}. The first picture shows the cone~$C_y$ in the~$d=2$ case. The second picture shows the two-dimensional section that arises in the construction of~$e'$.}
    \label{TheConeC1}
\end{figure}
\paragraph{The cone~$C_y$ does not intersect~$\OmOne$.} Let us prove this claim. Since the tangent plane to~$\OmOne$ at the point~$(y_{\bar d},h(y_{\bar d}))$ is described by the equation~$z_d - y_d = \scalprod{\nabla h (y_{\bar d})}{z_{\bar d} - y_{\bar d}}$, it suffices to prove the inequality
\eq{
z_d - y_d \leq \scalprod{\nabla h (y_{\bar d})}{z_{\bar d} - y_{\bar d}},\qquad z\in C_y.
}
We estimate
\eq{
z_d - y_d \Lref{\scriptscriptstyle z\in C_y} - C|y_{\bar d}||z_{\bar d} - y_{\bar d}| \leq -|\nabla h(y_{\bar d})||z_{\bar d} - y_{\bar d}| \leq \scalprod{\nabla h (y_{\bar d})}{z_{\bar d} - y_{\bar d}}. 
}
The inequality~$|\nabla h(y_{\bar d})| \leq C|y_{\bar d}|$ for all~$x$ in a compact set and all~$y$ in a neighborhood of~$x$ follows from the~$C^2$-smoothness assumption about~$\FreeBoundary\Om$. Thus,~$C_y$ indeed does not intersect~$\OmOne$.

A similar reasoning shows that
\eq{\label{Z_dGeq0}
|z-y| \lesssim |y_{\bar d}|, \quad \text{provided}\quad z\in C_y \ \text{and}\ z_d \geq 0.
}
In particular, in such a case~$z$ cannot belong to~$\FixedBoundary\Om$.

Let~$e_x$ be the point constructed in Proposition~\ref{C2Bound}. We consider two cases:~$e_x\in C_y$ and~$e_x \notin C_y$. 

\paragraph{Case~$e_x \in C_y$.} In this case,~$e_x \in \Vis_y$ and the segment~$[e_x,y]$ intersects~$\{z\mid z_d=0\}$ (since~$e_x$ lies below the~$z_d=0$ plane by~\eqref{Z_dGeq0}). Denote the point of intersection by~$P$. Then, the function~$G$ defined by
\eq{\label{FunctionG}
G(z) = \BG(z) - \BG(0) - L[\BG,0](z),\qquad z\in \Om,
}
is concave on this segment, attains the value~$0$ at~$e_x$ and is non-positive at~$P$. Thus, it is non-positive at~$y$ as well, and we have proved~$\BG(y) \leq \BG(x)+ L[\BG,x](y-x)$, which is stronger than~\eqref{C2Estimate2}.

\paragraph{Consider the case~$e_x\notin C_y$.} There exists a point~$e'\in \FixedBoundary\Om \cap C_y$ such that~$|e_x - e'| \lesssim |y_{\bar d}|$. For example, such a point can be found in the two-dimensional plane that passes through $y$,~$e_x$, and is orthogonal to the plane~$z_d=0$ (here we use Proposition~\ref{BoundedVisibility}; see also the second drawing on Fig.~\ref{TheConeC1}). The segment~$[e',y]$ intersects the plane~$\{z\mid z_d=0\}$ at the point~$P$. By~\eqref{Z_dGeq0},
\eq{\label{Eq829}
|y - P| \lesssim |y_{\bar d}|\quad \text{and}\quad 
|y - e'| \gtrsim 1.
}

By the concavity of~$G$ (see~\eqref{FunctionG} for the definition of~$G$),
\eq{
G(P) \geq \frac{|P-e'|}{|e'-y|}G(y) + \frac{|P-y|}{|e'-y|}G(e').
}
Since~$G(P) \leq0$,
\eq{
G(y) \leq-\frac{|P-y|}{|P-e'|}G(e')\Lseqref{C2Estimate} \frac{|P-y|}{|P-e'|} |e'-e_x|^2\Lseqref{Eq829}|y_{\bar d}|\cdot|y_{\bar d}|^2 = |y_{\bar d}|^3.
}
\end{proof}
\begin{Cor}\label{QuadratischeCor}
Let~$\Omega$ be a lens that satisfies~\eqref{StrictConvexity} and~\eqref{ConeCondition}. Assume that the boundaries of~$\Omega$ are~$C^2$-smooth. Let~$f\colon \FixedBoundary\Omega\to \R$ be a~$C^2$-smooth function such that~$\BG_{\Omega,f}$ is finite. For any~$\eps >0$ there exists a function~$G\colon\Om\to \R$ that is locally concave on~$\Om$ and satisfies the inequalities
\eq{\
\BG_{\Om,f}(x) \leq G(x) \leq \BG_{\Om,f}(x) + \eps,\qquad x\in \Om;
}
moreover\textup, for any~$x\in \FreeBoundary\Om$ there exists a linear function~$L[G,x]$ such that for any compact set~$K\subset \FreeBoundary\Om$ the inequality
\eq{\label{Quadratische}
G(y)\leq G(x) + L[G,x](y-x) - c_K|x-y|^2,\qquad x,y\in K,\quad |x-y| < \eps_K,
}
holds true with the positive constants~$c_K$ and~$\eps_K$ depending on~$K$ only. The coefficients of the linear functions~$L[G,x]$  are uniformly bounded when~$x$ runs through a compact subset of~$\FreeBoundary\Om$.
\end{Cor}
\begin{proof}
Let~$g\colon \Omega \to [0,1]$ be a strongly convex function (i.\,e., it is~$C^2$-smooth and with everywhere strictly positive definite Hesse matrix); such a function is easy to construct using~\eqref{StrictConvexity} and~\eqref{ConeCondition} (see Lemma~$4.5$ in~\cite{StolyarovZatitskiy2016}). We set
\eq{
G(x) = \BG_{\Omega,f}(x) + \eps g(x),
}
which leads to the natural choice of the linear functions~$L[G,x]$:
\eq{
L[G,x] = L[\BG,x] + \scalprod{\nabla g(x)}{\cdot},\quad x\in \FreeBoundary\Om.
}
The inequality~\eqref{Quadratische} then follows from Proposition~\ref{CubicProp}.
\end{proof}
\subsection{Construction of the extension}
Before we pass to the details, we will survey our method for construction of extensions. Assume~$G$ is a locally concave function on some set~$\omega$ and assume its superdifferential is non-empty at each point~$y\in \omega$. Let~$\tom$ be a set such that~$\omega \subset\tom$. Then, we can extend~$G$ to~$\tom$ by the formula
\eq{\label{LightExtension}
\tG(x) = \inf\Set{G(y) + L(x-y)}{y\in \omega,\ L\in \eth G|_y \text{ and } y\in \Vis_{x}^{\tom}},\quad x \in\tom.
}
The formula is not completely rigorous, because sometimes it is convenient to take all linear functions from the superdifferential of~$G$ at~$y$, sometimes we may pick only one. There are two questions concerning formula~\eqref{LightExtension}: when do we obtain a finite function and when is it locally concave on~$\tom$? We will give two simple sufficient conditions. 

The function~$\tG$ is finite when all~$L \in \eth G|_y$ are uniformly bounded when~$y$ runs through a compact set and for any~$x\in \tom$, the set~$\Vis_{x}^{\tom}$ is compact. 

The function~$\tG$ is locally concave provided for any~$x\in\tom$ there exists a neighborhood~$U_x$ (in the relative topology of~$\tom$) and a point~$y\in \om$ such that~$y\in \Vis_z^{\tom}$ for any~$z\in U_x$ and
\eq{
\tG(x) = G(y) + L(y-x)\quad \text{ for some } L\in \eth G|_y.
} 
Then~$\eth G|_x \ne \varnothing$ and the local concavity follows from Fact~\ref{Fact_superdiff}.

We will also need two theorems about approximation. For more details and proofs, see~\cite{AzagraStolyarov2022} (based on earlier work in~\cite{Azagra2013}).
\begin{Th}\label{OuterApproximation}
Let~$\Om$ be a non-empty strictly convex open proper subset of~$\R^d$. Let~$U$ be an open set that contains~$\cl \Om$. There exists another open set~$\Om'$  such that~$\cl\Om \subset \Om'$\textup,~$\cl\Om' \subset U$\textup, and~$\Om'$ is a strictly convex set with real-analytic boundary.
\end{Th}
\begin{Th}\label{InnerApproximation}
Let~$\Om$ be a non-empty strictly convex open proper subset of~$\R^d$. Let~$V$ be a closed set that lies inside~$\Om$. There exists another open set~$\Om'$ such that~$V \subset \Om'\subset \cl \Om' \subset \Om$ and~$\Om'$ is a strictly convex set with real-analytic boundary.
\end{Th}

\begin{St}\label{LightApproximation}
Let~$\Om$ be a non-empty strictly convex open proper subset of~$\R^d$. Let~$\rho\colon \partial\Om \to (0,1]$ be a continuous function. There exists a strictly convex open set~$\Om'$ with real-analytic boundary such that~$\cl\Om' \subset \Om$ and if~$x,y\in \partial\Om$ see each other in~$\cl\Om\setminus\Om',$ then~$|x-y| < \rho(x)$.
\end{St}
\begin{proof}
Consider the set
\eq{
V = \Set{z\in \cl \Om}{z = \frac{x+y}{2},\quad x,y\in\partial \Om,\ |x-y|\geq \rho(x)}.
}
This set is closed and lies inside~$\Om$. If two points~$x,y\in \partial \Om$ see each other in~$\cl\Om\setminus V$, then~$|x-y| < \rho(x)$. It remains to apply Theorem~\ref{InnerApproximation} to replace~$V$ with a larger set~$\Om'$.
\end{proof}
The following corollary is obtained by combination of Corollary~\ref{QuadratischeCor} and Proposition~\ref{LightApproximation}.
\begin{Cor}\label{QuadratischeCorLight}
Let~$\Omega$ be a lens that satisfies~\eqref{StrictConvexity} and~\eqref{ConeCondition}. Assume that the boundaries of~$\Omega$ are~$C^2$-smooth. Let~$f\colon \FixedBoundary\Omega\to \R$ be a~$C^2$-smooth function such that~$\BG_{\Omega,f}$ is finite. For any~$\eps >0$ there exists an extension~$\Om'$ whose boundaries are~$C^2$-smooth and a function~$G\colon\Om\to \R$ that is locally concave on~$\Om$ and satisfies the inequalities
\eq{\
\BG_{\Om,f}(x) \leq G(x) \leq \BG_{\Om,f}(x) + \eps,\qquad x\in \Om;
}
moreover\textup, for any~$x\in \FreeBoundary\Om$ there exists a linear function~$L[G,x]$ such that for any compact set~$K \subset \FreeBoundary\Om$ there exists~$c_K > 0$ such that~\eqref{Quadratische} holds true whenever~$x,y \in K$ see each other in~$\Om'$.  The coefficients of the linear functions~$L[G,x]$ are uniformly bounded when~$x$ runs through a compact subset of~$\FreeBoundary\Om$.
\end{Cor}
Now we are ready to define our extension~$\tG$ by a formula similar to~\eqref{LightExtension}:
\eq{\label{LightExtensionOut}
\tG(z) = \begin{cases}
G(z),\quad &z\in \Om;\\
\inf\set{G(y)+L[G,y](z-y)}{y\in \Vis_{z}^{\Om'}\cap\FreeBoundary\Om},\quad &z\in\Om'\setminus\Om.
\end{cases}
}
\begin{Le}\label{SdLemma}
For any~$x\in\FreeBoundary\Om$ there exists a relatively open set~$U_x\subset \Om'\setminus\interior\Om$ that contains~$x$ and such that for any~$z\in U_x$ the value~$\tG(z)$ is finite and the superdifferential of~$\tG$ at~$z$ is non-empty. 
\end{Le}
\begin{figure}
\hspace{2,5cm}
\includegraphics[width=0.7\textwidth]{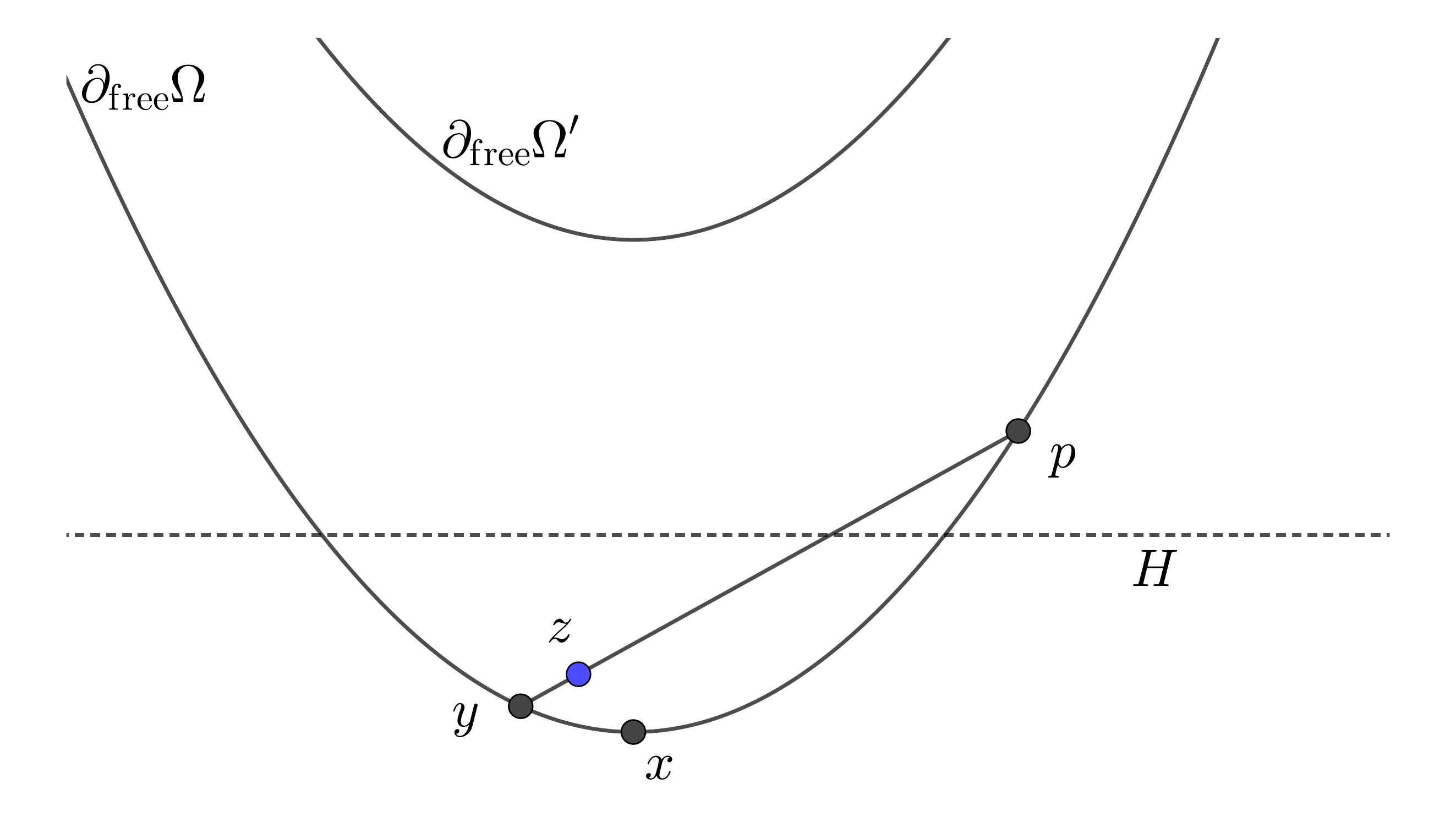}
\caption{Illustration to the proof of Lemma~\ref{SdLemma}.}
\label{Ill9}
\end{figure}
\begin{proof}
Let~$H$ be a hyperplane that separates~$x$ from~$\FreeBoundary \Om'$ (say,~$H$ is closer to~$x$ than to the latter set). It suffices to prove that if~$z\in \Om'\setminus\interior \Om$ is sufficiently close to~$x$, then the infimum in~\eqref{LightExtensionOut} is attained at~$y$ that lies on the same side of~$H$ as~$x$. See Fig.~\ref{Ill9} for visualization. Then~$y$ sees a neighborhood of~$z$ and~$L[G,y]$ belongs to the superdifferential of~$\tG$ at~$z$.  

Let~$p\in\FreeBoundary\Om$ be a point that lies on the other side of~$H$ than~$x$ and that sees~$z$ in~$\Om'$. We wish to prove that there exists~$y$ on the same side of~$H$ as~$x$  such that
\eq{\label{WantedSuperDif}
G(y) + L[G,y](z-y) \leq G(p) + L[G,p](z-p),
} 
provided~$z$ is sufficiently close to~$x$. Let~$y$ be the intersection of the line passing through~$z$ and~$p$ with~$\FreeBoundary\Om$ lying on the same side of~$H$ as~$x$. Let also~$K \subset \FreeBoundary\Om$ be a compact set that contains~$x$ with all the points it can see in~$\Om'$ and all the points the latter points can see. Let~$M$ be the supremum of the Lipschitz constants of the functions~$L[G,q]$ when~$q\in K$. Then,
\mlt{
G(y) + L[G,y](z-y) \leq G(y) + M|z-y| \Leqref{Quadratische} G(p) + L[G,p](y-p) - c_{K}|p-y|^2 + M|z-y| \leq\\ G(p) + L[G,p](z-p) - c_K|p-y|^2 + 2M|z-y|.
} 
We see that~\eqref{WantedSuperDif} indeed holds true provided~$z$ is sufficiently close to~$x$, because in this case~$y$ is also sufficiently close to~$x$ while~$|p-y|$ is bounded away from zero.
\end{proof}
\begin{proof}[Proof of Theorem~\ref{Extension}.]
We construct the set~$V\subset \Om_1$ as the complement to the union of the sets~$U_x$,~$x\in\FreeBoundary\Om$, provided by Lemma~\ref{SdLemma}. We apply Theorem~\ref{InnerApproximation} and obtain a strictly convex set~$\tOmOne$ such that~$V \subset \tOmOne\subset \cl \tOmOne\subset \OmOne$. Set~$\tOm = \cl\OmNull \setminus \tOmOne$. Then,~$\tG$ constructed by~\eqref{LightExtensionOut} is locally concave on~$\tOm$ (it is locally concave by Fact~\ref{Fact_superdiff}). Therefore,~$\BG_{\tOm,f}(z) \leq \tG(z)$ for any~$z\in \tOm$ and
\eq{
\BG_{\Om,f}(x)\leq\BG_{\tOm,f}(x) \leq G(x) \leq \BG_{\Om,f}(x) + \eps,\qquad x\in \Om.
}
\end{proof}

\subsection{Proof of Theorem~\ref{ExtensionGen}}
In order to prove Theorem~\ref{ExtensionGen},  we will need to extend a locally convex function via formula~\eqref{LightExtension} over the fixed boundary.
\begin{Le}\label{LocalConcavityOfOuterExtension}
Let~$\Om$ be a lens that satisfies the requirements~\eqref{StrictConvexity} and~\eqref{ConeCondition}. Let~$\omega \supset \Om$ be a set whose interior contains~$\Om \setminus \FreeBoundary\Om$. Let~$G\colon \omega\to\R$ be a locally Lipschitz function that is locally concave on~$\Om$ and has non-empty superdifferential at each point. Assume~$\Om'$ is an open convex set that contains~$\Om$ and such that each~$x\in \Om'$ sees only a compact subset of~$\FixedBoundary\Om$ in~$\cl\Om'\setminus\OmNull$. Then\textup, the function
\eq{\label{LightExtensionOut2}
\tG(x) = \begin{cases}
G(x),\quad &x\in \Om;\\
\inf\Set{G(y) + L(x-y)}{y\in \Om\cap \Vis_{x}^{\Om'\setminus\OmOne}, L\in \eth G|_y},\quad&x\in \Om'\setminus\OmNull,
\end{cases}
}
is finite and locally concave on~$\Om'\setminus\OmOne$.
\end{Le}
\begin{proof}
By local concavity of~$G$, we may consider only~$y\in \FixedBoundary\Om \cap \Vis_{x}^{\Om'\setminus\OmNull}$ when calculating the infimum in~\eqref{LightExtensionOut2}. Since the sets~$\FixedBoundary\Om \cap \Vis_{x}^{\Om'\setminus\OmNull}$ are compact and the function~$G$ is locally Lipschitz on~$\omega$,~$\tG$ does not attain the value~$-\infty$.  Thus, it remains to verify the local concavity of~$\tG$. For that, we will prove~$\tG$ has a non-empty superdifferential at each point~$x\in \Om'\setminus\OmNull$.

Let~$x\in \Om'\setminus\OmNull$ and assume~$\tG(x) = G(y) + L(x-y)$,~$y\in  \FixedBoundary\Om \cap \Vis_{x}^{\Om'\setminus\OmNull}$ and~$L\in \eth G|_y$. It suffices to prove~$\tG(z)\leq G(y) + L(z-y)$ when~$z$ lies in a sufficiently small neighborhood of~$x$. This is true since~$z$ sees~$y$ in~$\Om'\setminus\OmOne$ (now we are using the initial formula~\eqref{LightExtensionOut2}). The reasoning in the case when the infimum that defines~$\tG(x)$ is attained at a sequence of~$y_n$ does not differ. 
\end{proof}
\begin{Fact}\label{CanSeeCompact}
Let~$\OmNull$ be a strictly convex open set\textup, let~$x$ be a point of its boundary. There exists a neighborhood~$U_x$ in~$\R^d\setminus \OmNull$ such that any point~$y\in U_x$ can see only a compact subset of~$\partial \OmNull$ in~$\R^d\setminus\OmNull$.
\end{Fact}
\begin{proof}[Proof of Theorem~\ref{ExtensionGen}]
Fix~$p\in \Om\setminus \FixedBoundary\Omega$ and~$\eps > 0$.

First, we construct an open strictly convex set~$\Om_{-1}$ with real analytic boundary such that it contains the closure of~$\OmNull$ and any point~$x\in \cl\Om_{-1}\setminus\OmNull$ can see only a compact subset of~$\FixedBoundary\Om$ in~$\Om_{-1}\setminus \OmNull$; this is done by a combination of Theorem~\ref{OuterApproximation} and Fact~\ref{CanSeeCompact}.

Second, we construct a strictly convex set~$\Om'$ such that~$\cl\Om'\subset \OmNull$,~$\cl\OmOne\subset \Om'$,~$p\in \Om'$, and if~$y$ and~$z$ from~$\FixedBoundary\Om$ see each other in~$\cl\OmNull\setminus\Om'$, then~$|y-z|\leq 1$; this is done by application of Proposition~\ref{LightApproximation} with~$\rho\equiv \delta$, where~$\delta$ is sufficiently small.

Note that any point~$x\in \cl\Om_{-1}\setminus\Om'$ can see only a compact subset of~$\partial\Om'$ in~$\cl\Om_{-1}\setminus \Om'$. We also note that the restriction of the function~$\BG_{\Om,f}$ to the set~$\Om'$ is locally Lipschitz and has a non-empty superdifferential at each point; the set~$\set{\nabla L}{L \in \eth \BG_{\Om,f}|_x, x\in K}$ is uniformly bounded for any compact set~$K\subset \Om'$. Thus, we may construct the function~$G\colon \cl\Om_{-1}\setminus \OmOne\to \R$ by the formula
\eq{
G(x) = \begin{cases}
\BG_{\Om,f}(x),\quad &x\in \cl\Om'\setminus\OmOne;\\
\inf\Set{\BG_{\Omega,f}(y)+L(x-y)}{y\in \Om'\cap \Vis_{x}^{\cl\Om_{-1}\setminus\OmOne}, L\in \eth \BG_{\Om,f}|_y},\quad&x\in \cl\Om_{-1}\setminus\Om'.
\end{cases}
}
By Lemma~\ref{LocalConcavityOfOuterExtension},~$G$ is a locally concave function. What is more,~$G$ is continuous on~$\partial \Om_{-1}$. Therefore, the function
\eq{
\BG_{\cl\Om_{-1}\setminus\OmOne, G|_{\partial \Om_{-1}}}
}
may be 'extended' through the free boundary by Theorem~\ref{Extension}. Let~$\tOm$ be the extension of~$\cl\Om_{-1}\setminus \OmOne$, let~$\tG$ be the 'extended' function. Then,
\eq{
\tG(p) \leq \BG_{\cl\Om_{-1}\setminus\OmOne, G|_{\partial \Om_{-1}}}(p) + \eps \leq G(p) + \eps  = \BG_{\Om,f}(p) + \eps.
}
\end{proof}

\section{Limitations of current methods}\label{S9}
The lemma below justifies the appearance of the auxiliary domain~$\hOmOne$ in~\eqref{ClassC}. It also shows why we are not able to prove~\eqref{CoincidenceFormula} for the points~$x\in \FreeBoundary\Omega$ in Theorem~\ref{IntervalTheorem}. We recall the definition of a cheese domain from~\cite{StolyarovZatitskiy2016}. A set~$\Omega\subset \R^d$ is called a \emph{cheese domain} if it may be represented as
\begin{equation}\label{RepresentationForCheeseDomain}
\Omega =\cl \Omega_0 \setminus \bigcup_{j=1}^N \Omega_j,
\end{equation}
where the domains~$\Omega_j$,~$j=0,1,\ldots,N,$ are strictly convex, open, and bounded; the `holes'~$\Omega_j$,~$j=1,2,\ldots,N$, are mutually separated and also lie inside the interior of~$\Omega_0$.
\begin{Le}\label{CheeseBoundary}
Let~$\Omega\subset \R^2$ be a cheese domain such that the domains~$\Omega_j,$~$j = 1,2,\ldots,N,$ in the representation~\eqref{RepresentationForCheeseDomain} have~$C^1$-smooth boundaries. Let~$\varphi \colon \mathbb{T} \to \partial \Omega_0$ be such that for any arc~$J \subset \mathbb{T}$ the point~$\av{\varphi}{J}$ lies in~$\Omega$. If~$\av{\varphi}{\mathbb{T}} \in \partial \Omega_j$ for some~$j,$ then~$\varphi$ attains two values.
\end{Le}
\begin{proof}
Without loss of generality,~$\av{\varphi}{\mathbb{T}} \in \partial \Omega_1$. Let~$\ell$ be the tangent to~$\partial \Omega_1$ at the point~$x=\av{\varphi}{\mathbb{T}}$. We say that a point lies below~$\ell$ if it is strictly separated by~$\ell$ from~$\Omega_1$ and say that it lies above~$\ell$ in the case the point and~$\OmOne$ belong to the same open half-plane generated by~$\ell$. 

We will first prove that~$\varphi$ does not attain values on the part of~$\partial \Omega_0$ lying below~$\ell$. Assume the contrary. Let~$t$ be a Lebesgue point of~$\varphi$ such that~$\varphi(t)$ lies below~$\ell$. Consider the point
\eq{\label{92}
\av{\varphi}{\mathbb{T}\setminus [t-\eps,t+\eps]} = \frac{1}{1-2\eps}\Big(x - 2\eps\av{\varphi}{[t-\eps,t+\eps]}\Big). 
}
When~$\eps$ is sufficiently close to~$0$, the point~$\av{\varphi}{[t-\eps,t+\eps]}$ lies close to~$\varphi(t)$. Thus, the point~\eqref{92} lies in a small neighborhood of~$x$ and above~$\ell$. In particular, it does not belong to~$\Omega$ provided~$\eps$ is sufficiently small. This is a contradiction.

Second, we will prove that~$\varphi$ does not attain values on the part of~$\partial \Omega_0$ lying above~$\ell$. Consider an affine function~$L$ that is equal to zero on~$\ell$ and is negative below~$\ell$. Then,~$L(x) = 0$. On the other hand,
\begin{equation*}
L(x) = \int\limits_{\mathbb{T}}L\big(\varphi(t)\big)\,dt > 0
\end{equation*}
provided~$L(\varphi) > 0$ on a set of non-zero measure (since we have proved that~$L(\varphi)$ is always non-negative).
\end{proof}

Finally we will show that the convexity of~$\OmOne$ is necessary for~\eqref{CoincidenceFormula} (note that the definition of the minimal locally concave function~$\BG_{\Om,f}$ and the Bellman function~$\Bell_{\Om,f}$ do not require convexity of~$\OmOne$). 
\begin{figure}
\begin{center}
\includegraphics[width=0.6\textwidth]{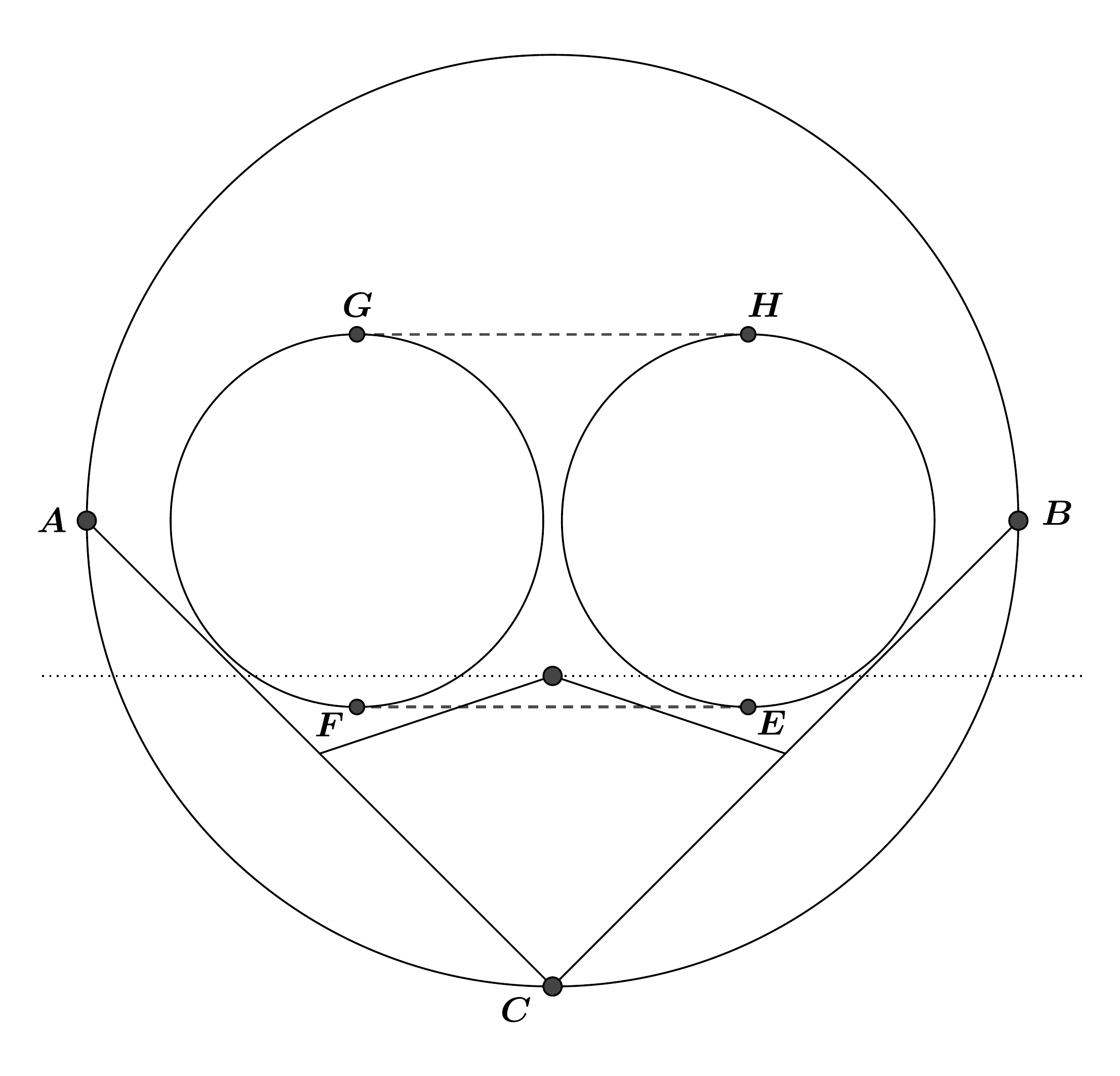}
\caption{The domain~$\Om$ and the set of Bellman points of~$\varphi$.}
\end{center}
\label{AB}
\end{figure}

We will shortly construct an example of~$\Omega$ and~$f$ such that~$\Bell_{\Omega,f} > \BG_{\Omega,f}$. Let~$\Omega_0$ be the unit circle. Consider three points
\begin{equation*}
A = (-1,0),\quad B = (1,0),\quad C = (0,-1).
\end{equation*}
Define the function~$\varphi \colon [0,1]\to \partial \Omega_0$ by the rule:
\begin{equation*}
\varphi(t) = \begin{cases}
A,\quad & [t]\in [0,\frac13);\\
C, \quad & [t] \in [\frac13,\frac23);\\
B,\quad & [t] \in [\frac23,1].
\end{cases}
\end{equation*}
\begin{Fact}\label{DescriptionOfBellmanPoints}
The set of Bellman points of~$\varphi$\textup, i.\,e.\textup, the points~$\av{\varphi}{\J}$\textup,  is 
\begin{equation}\label{BellmanPointsForExample}
AC\cup BC \cup \{x = \alpha A + \beta B + \gamma C\mid \alpha + \beta + \gamma = 1, \gamma \geq \alpha \geq 0, \gamma \geq \beta \geq 0\}.
\end{equation}
\end{Fact}
Now we construct the domain~$\Omega_1$ and the function~$f$. Define~$\Omega_1$ by the rule
\begin{equation*}
\Omega_1 = \Big\{x\in\mathbb{R}^2\,\Big|\; \Big\|x- \Big(-\frac12,0\Big)\Big\| \leq \frac{1}{2.9}\Big\}\cup \Big\{x\in\mathbb{R}^2\,\Big|\; \Big\|x- \Big(\frac12,0\Big)\Big\| \leq \frac{1}{2.9}\Big\}.
\end{equation*}
We draw Figure~\ref{AB} for reader's convenience (the picture has slightly different numeric parameters, what is important is that the two `erased' circles do not intersect the set~\eqref{BellmanPointsForExample}, however, the average of~$\varphi$ lies above the lower common tangent to the circles).  

By Fact~\ref{DescriptionOfBellmanPoints},~$\varphi \in \Class (\Omega)$. Define~$f\colon \FixedBoundary\Om \to \mathbb{R}$ by the formula
\begin{equation}\label{BoundaryCondition}
f(y) = \begin{cases}
0,\quad &y_2 \leq 0;\\
-y_2,\quad &y_2 > 0.
\end{cases} 
\end{equation}
Clearly,~$\av{f(\varphi)}{[0,1]} = 0$, and, thus,~$\Bell_f(0,-\frac13) \geq 0$. 

\begin{Le}
Consider~$f$ given by~\eqref{BoundaryCondition}. Then\textup,~$\BG_{\Om,f}(0,-\frac13) < 0$.
\end{Le}
 \begin{proof}
It is clear that~$\BG_{\Om,f}(y) \leq f(y)$ provided we extend the function~$f$ to~$\Omega$ by the same formula~\eqref{BoundaryCondition}. Consider the points~$E$ and~$F$ on the erased circles that have the smallest possible second coordinates. In other words, they are the points on the lower common tangent to these two circles. Similarly, let~$G$ and~$H$ be the points having the largest possible second coordinates. 

Let~$L$ be the linear function that coincides with~$f$ on the segments~$EF$ and~$GH$ and let the part of~$\Omega$ lying between~$EF$ and~$GH$ be called the channel. Define the new function~$\Phi$ to be equal~$L$ on the channel and to be equal~$f$ everywhere else. It is easy to observe that~$\Phi$ is locally concave. On the other hand,~$L < f$ and thus,~$\Phi(0,-\frac13) < 0$.
\end{proof}

\bibliography{mybib}{}
\bibliographystyle{amsplain}

\bigskip

Dmitriy Stolyarov, d.m.stolyarov@spbu.ru,

St. Petersburg State University, Department of Mathematics and Computer Science;

\medskip

Pavel Zatitskiy, pavelz@pdmi.ras.ru,

St. Petersburg State University, Department of Mathematics and Computer Science.

\end{document}